\documentclass{elsarticle}

\usepackage{amsfonts}
\usepackage{amsthm}
\usepackage{amsmath}
\usepackage{amssymb}
\usepackage{mathtools} %for prescript
\usepackage{xcolor}
\usepackage{caption}
\usepackage{subcaption}
\usepackage{graphicx}
\usepackage{float}
\usepackage{soul} %for strikethrough
\usepackage{mathrsfs}

\usepackage[top=1in, bottom=1in, left=1in, right=1in]{geometry}

\newtheorem{thm}{Theorem}[section]
\newtheorem{lem}{Lemma}[section]
\newtheorem{coroll}{Corollary}[section]

\theoremstyle{definition}
\newtheorem{defn}{Definition}[section]

\theoremstyle{remark}
\newtheorem{remark}{Remark}[section]

\definecolor{cucol}{rgb}{0,0,0.8}
\definecolor{afcol}{rgb}{1,0,0}

\numberwithin{equation}{section}

\begin{document}
	
	\begin{frontmatter}
		
		\title{Fractional Leibniz integral rules for Riemann-Liouville and Caputo fractional derivatives and their applications}
		
		%%this line removes the date, but space is still left for it;
		%if used, remove the \vspace{-1cm}
		\date{}
		
		%this gives the date in the form Mon 30 Jan 2012, 8:57pm;
		%if used, retain the \vspace{-1cm}
		%\date{\shortdayofweekname{\day}{\month}{\year}{ }\mydate\today}
		\author[]{Ismail T. Huseynov }
		\ead{ismail.huseynov@emu.edu.tr}
		\author[]{Arzu Ahmadova }
		\ead{arzu.ahmadova@emu.edu.tr}
		\author[]{Nazim I. Mahmudov \corref{cor1}}
		\ead{nazim.mahmudov@emu.edu.tr}
		\cortext[cor1]{Corresponding author}

		%\author[]{\corref{cor1}}
		%\ead{}
		%\author[]{}
		%\ead{}

		\address { Department of Mathematics, Faculty of Arts and Sciences, Eastern Mediterranean University, Mersin 10, Gazimagusa, TRNC, Turkey}
		
		% Latex won't make the title unless told:
		
		%\maketitle
		
		%%to remove the space left for date, use:
		
	\begin{abstract}
		\noindent In recent years, the theory for Leibniz integral rule in the fractional sense has not been able
		to get substantial development. As an urgent problem to be solved, we study a  Leibniz integral rule for Riemann-Liouville and Caputo type  differentiation operators with general fractional-order of $n-1 <\alpha \leq n$, $n \in \mathbb{N}$ . A rule of fractional differentiation under integral sign with general order is necessary and applicable tool for verification by substitution for candidate solutions of inhomogeneous multi-term fractional differential equations. We derive explicit analytical solutions of generalized Bagley-Torvik equations in terms of recently defined bivariate Mittag-Leffler type functions that based on fractional Green's function method and verified solutions by substitution in accordance by applying the  fractional Leibniz integral rule. Furthermore, we study an oscillator equation as a special case of differential equations with multi-orders via the  Leibniz integral rule.
		 
	\end{abstract}
	
	
\end{frontmatter}
\section{Introduction}

Fractional calculus is a generalization of the classical differential calculus which has attracted growing attention due to the applications for many problems in science and engineering such as reaction-diffusion systems \cite{Datsko-Gafiychuk}, viscoelasticity \cite{Schiessel-Metzler-Blumen-Nonnenmacher}, electrical circuits \cite{Ahmadova-Mahmudov, Kaczorek}, control theory \cite{Mahmudov}, stochastic analysis \cite{Ahmadova A-Mahmudov N} and time-delay systems \cite{Huseynov-Mahmudov}.

One of the most frequently encountered tools in the theory of fractional calculus is furnished by the Riemann-Liouvulle $\prescript{RL}{t_{0}}{D}^{\alpha}_{t}$ and Caputo $\prescript{C}{t_{0}}{D}^{\alpha}_{t}$ fractional differentiation operators. A fractional analogue of Leibniz rule for differentiation is crucial and useful properties of these operators. Podlubny in \cite{I. Podlubny}, Baleanu and Trujillo in \cite{Baleanu-Turijillo} give a proof the Leibniz rule for Riemann-Liouville and Caputo type derivatives; these results are stated respectively below.

Assume that $0<\alpha < 1$ and $f,g:[t_{0},T]\subset\mathbb{R}\to\mathbb{R}$ with all their derivatives
are continuous. Then

\begin{align}
&\prescript{RL}{t_{0}}{D}^{\alpha}_{t}\left\lbrace f(t)g(t) \right\rbrace=
\sum_{k=0}^{\infty}\binom{\alpha}{k}f^{(k)}(t)\prescript{RL}{t_{0}}{D}^{\alpha-k}_{t}g(t),\quad t \in(t_{0},T),\\
&\prescript{C}{t_{0}}{D}^{\alpha}_{t}\left\lbrace f(t)g(t) \right\rbrace=
\sum_{k=0}^{\infty}\binom{\alpha}{k}f^{(k)}(t)\prescript{RL}{t_{0}}{D}^{\alpha-k}_{t}g(t)-\frac{(t-t_{0})^{-\alpha}}{\Gamma(1-\alpha)}f(t_0)g(t_0),\quad t \in(t_{0},T),
\end{align}
where binomial coefficients satisfy the identity:

\begin{equation*}
\binom{\alpha}{0}=1, \quad\binom{\alpha}{k}=\frac{\alpha(\alpha-1)\ldots(\alpha-k+1)}{k!}.
\end{equation*}

Another essential property of Riemann-Liouville fractional differentiation operator is obtained by Podlubny in \cite{I. Podlubny} which is called  fractional Leibniz integral rule stated as below:
\begin{equation}\label{Leibniz-integral}
\prescript{RL}{t_{0}}{D}^{\alpha}_{t}\int\limits_{0}^{t}K(t,\tau)\mathrm{d}\tau= \lim\limits_{\tau\to t-0}\prescript{RL}{\tau}{D}^{\alpha-1}_{t}K(t,\tau)+\int\limits_{0}^{t}\prescript{RL}{\tau}{D}^{\alpha}_{t}K(t,\tau)\mathrm{d}\tau, \quad \alpha \in(0,1),
\end{equation}
where lower terminal $t_{0}=0$.

The following important particular case for convolution operator whenever we have $K(t-\tau)f(\tau)$ instead of $K(t,\tau)$, the relationship \eqref{Leibniz-integral} takes the form:

\begin{equation}\label{Leibniz-integral-convolution}
\prescript{RL}{0}{D}^{\alpha}_{t}\int\limits_{0}^{t}K(t-\tau)f(\tau)\mathrm{d}\tau= \lim\limits_{\tau\to t-0}f(t-\tau)\prescript{RL}{\tau}{D}^{\alpha-1}_{t}K(\tau)+\int\limits_{0}^{t}\prescript{RL}{0}{D}^{\alpha}_{\tau}K(\tau)f(t-\tau)\mathrm{d}\tau, \quad \alpha \in(0,1).
\end{equation}

It is important to note that the above tools are necessary for checking by substitution method for fractional differential equations with variable and constant coefficients.

Fractional differential equations (FDEs) are differential equations involving derivatives of arbitrary (fractional) order. FDEs provide one of the most accurate tools to describe hereditary properties of natural phenomenon. Using fractional derivatives instead of integer-order derivatives allows us for the modeling of a wider variety of behaviours.
However, sometimes, FDEs involving one fractional order of differentiation are not sufficient to demonstrate physical processes. Therefore, recently, several authors have studied more general types of fractional-order models, such as multi-term equations \cite{Hilfer-Luchko-Tomovski,Luchko-Gorenflo,Bazhlekova,Pak-Choi-Sin-Ri,Diethelm-Ford} and multi-dimensional systems \cite{Kaczorek,Ahmadova A-Mahmudov N,Huseynov-Mahmudov,Bonilla-Rivero-Trujillo,Diethelm-Siegmund-Tuan,Kaczorek-Idczak,H-A-F-M}.

Multi-term differential equations with fractional-order have been studied 
and solved using various mathematical methods, of which we mention a few as follows. Luchko and several collaborates \cite{Hilfer-Luchko-Tomovski,Luchko-Gorenflo} have used the method of operational calculus to solve multi-term FDEs with constant coefficients with regard to various types of fractional derivatives. Bazhlekova \cite{Bazhlekova} has considered multi-term fractional relaxation equations with Caputo fractional derivatives by using a Laplace transform technique, and studied the fundamental and impulse-response solutions of the initial value problem (IVP). Kaczorek and Idczak \cite{Kaczorek-Idczak} have considered existence and uniqueness results and a Cauchy formula for the analytical solution of  the time-varying linear system with Caputo fractional derivative. Pak et al. \cite{Pak-Choi-Sin-Ri} has recently investigated multi-term FDEs with variable coefficients using a new method to construct analytical solutions.

As one of the important special cases of multi-term FDEs, Bagley-Torvik equations have been discussed in terms of analytical \cite{I. Podlubny,Bagley-Torvik,M-H-A-A,Wang-Wang} and numerical methods \cite{Kai-Ford,Srivastava et.al.}. Bagley-Torvik equations with $\frac{1}{2}$-order or $\frac{3}{2}$-order derivative describe the motion of real physical systems in a Newtonian fluid \cite{Bagley-Torvik}. In 1984, Bagley and Torvik \cite{Bagley-Torvik} have considered the following Cauchy problem under the homogeneous initial conditions:
\begin{align}\label{B-T}
&m y''(r)+\frac{2S\sqrt{\mu \rho}}{m}\left( \prescript{C}{}{D}^{\alpha}_{0_{+}}y\right)(r)+k y(r)=g(r), \quad r>0,\\
&y(0)=y'(0)=0 \nonumber,
\end{align}
where $\left( \prescript{C}{}{D}^{\alpha}_{0_{+}}y\right)(\cdot)$ is Caputo fractional differential operator of order $\alpha=\frac{1}{2}$ or $\alpha=\frac{3}{2}$, $S$- an area of the rigid plate, $\mu$- viscosity, $\rho$-fluid density, $m$-mass, $k$- spring of stiffness and $g(\cdot)$-an external force. An analytical solution of \eqref{B-T} has introduced by Podlubny \cite{I. Podlubny} in the form:
\begin{equation}
y(r)=\int\limits_{0}^{r}G(r-\tau)g(\tau)\mathrm{d}\tau, \quad r>0,
\end{equation} 
with
\begin{equation*}
G(r)=\frac{1}{m}\sum_{l=0}^{\infty}\frac{(-1)^{l}}{l!}\left( \frac{k}{m}\right)^{l}r^{2l+1}\mathcal{E}^{(l)}_{\frac{1}{2},2+\frac{3l}{2}}\left(\frac{-2S\sqrt{\mu \rho}}{m}\sqrt{r} \right), 
\end{equation*}
where $\mathcal{E}^{(l)}_{\alpha,\beta}(\cdot)$ is the $l$th-derivative of two-parameter Mittag-Leffler function.
In \cite{M-H-A-A}, Mahmudov et al. have studied explicit analytical solutions for several families of generalized multidimensional Bagley-Torvik equations with permutable matrices. In \cite{Wang-Wang}, Wang et al. have modified the following Bagley-Torvik equation
\begin{equation}\label{B-T-1}
y''(r)+\mu\left( \prescript{C}{}{D}^{\alpha}_{0_{+}}y\right)(r)+y(r)=0, \quad \mu,r>0,
\end{equation}
where $\alpha=\frac{1}{2}$ or $\alpha=\frac{3}{2}$, to the sequential FDEs and introduced a general solution of \eqref{B-T-1} by using the technique related to characteristic roots. The numerical point of view Diethelm and Ford in \cite{Kai-Ford} have used linear multi-steps , Srivastava et al. in \cite{Srivastava et.al.} have applied wavelet approach to obtain approximate solutions of the Bagley-Torvik equations.

Therefore, the plan of this paper is systematized as below. Section \ref{Sec:prel} is a mathematical preliminary section where we recall main definitions and results from fractional calculus, special functions and necessary lemmas from fractional differential equations. Section \ref{generalized Leibniz rule} is devoted to formulating the  Leibniz integral rule for higher order derivatives of Lebesgue integration which depends on parameter in classical sense. In Section \ref{Generealized fractional Leibniz rules}, we have introduced fractional differentiation under the integral sign in Riemann-Liouville and Caputo sense. Moreover, we have considered the derivative of convolution operator which has more importance for differential equations with classical or fractional order. In Section \ref{Green's function method}, we have acquired explicit analytical solutions of Bagley-Torvik equations with Riemann-Liouville and Caputo type fractional derivatives in terms of recently defined bivariate Mittag-Leffler type functions in accordance with fractional Green's function method and tested the candidate solutions by using our newly defined tools which are natural generalization of well-known Leibniz integral rule. At the end, in Section \ref{concl} we give the conclusions and future directions.

\section{Mathematical preliminaries}\label{Sec:prel}
We embark on this section by briefly introducing the essential structure of fractional calculus, special functions and fractional differential operators (for the more salient details on the matter, see the textbooks \cite{I. Podlubny,Diethelm,Gorenflo,Miller-Ross,Oldham-Spanier, Samko-Kilbas-Marichev}). We begin by defining some notations, Riemann-Liouville and Caputo fractional differentiation operators which are fundamental for fractional calculus and fractional differential equations.

Let $\mathbb{R}^{n}$ be Euclidean space and $\mathbb{J}$ be some interval of the real line, i.e. $\mathbb{J}\subset \mathbb{R}$. We suppose that $\mathbb{J}=[t_{0},T]$ for some $t\in \mathbb{\hat{J}}$ and denote $\mathbb{\hat{J}}=(t_{0},T)$. Assume that $f: \mathbb{J}\to \mathbb{R}$ is an absolutely continuous function.  

\begin{defn}[\label{defRL}\cite{Miller-Ross,Oldham-Spanier,Samko-Kilbas-Marichev}]

	The Riemann-Liouville derivative operator of fractional order $n-1<\alpha \leq n$ for $n\in\mathbb{N}$ is defined by
	\begin{equation}
	\left(\prescript{RL}{t_{0}}{D}^{\alpha}_{t}g\right) (t)=\frac{\mathrm{d}^n}{\mathrm{d}t^n}\left(\prescript{}{t_{0}}{I}^{n-\alpha}_{t}g\right) (t)=\frac{1}{\Gamma(n-\alpha)}\frac{\mathrm{d}^n}{\mathrm{d}t^n}\int\limits^{t}_{t_{0}}(t-s)^{n-\alpha-1}g(s)\mathrm{d}s, \quad t\in \mathbb{\hat{J}},
	\end{equation}
\end{defn}
where $\prescript{}{t_{0}}{I}^{\alpha}_{t}$ is the Riemann-Liouville integral operator of order $\alpha>0$ which is defined by
\begin{equation}
\left(\prescript{}{t_{0}}{I}^{\alpha}_{t}g\right) (t)=\frac{1}{\Gamma(\alpha)}\int\limits_{t_{0}}^{t}(t-s)^{\alpha-1}g(s)\mathrm{d}s, \quad t\in \mathbb{\hat{J}}.
\end{equation}
Furthermore, the following equality holds true:
\begin{equation}\label{R-L der-R-L int}
\left(\prescript{RL}{t_{0}}{D}^{\alpha}_{t}\left(\prescript{}{t_{0}}{I}^{\alpha}_{t}g\right)\right) (t)=g(t), \quad \alpha>0, \quad t \in \mathbb{\hat{J}}.
\end{equation}

\begin{defn}[\label{defCaputo}\cite{I. Podlubny,Diethelm,Kilbas-Srivastava-Trujillo}]
	 The Caputo derivative operator of fractional order $n-1<\alpha \leq n$ for $n\in\mathbb{N}$ is defined by
	\begin{equation}\label{Caputoder}
	\left( \prescript{C}{t_{0}}D^{\alpha}_{t}g\right) (t)=\prescript{}{t_{0}}I^{n-\alpha}_{t}\left(\frac{\mathrm{d}^n}{\mathrm{d}t^n}g\right)(t)=\frac{1}{\Gamma(n-\alpha)}\int\limits_{t_{0}}^{t}(t-s)^{n-\alpha-1}\frac{\mathrm{d}^n}{\mathrm{d}s^n}g(s)\mathrm{d}s, \quad t\in \mathbb{\hat{J}}.
	\end{equation}

	Moreover, the next relation holds true:
	\begin{equation}\label{C der-R-L int}
	\left(\prescript{C}{t_{0}}{D}^{\alpha}_{t}\left(\prescript{}{t_{0}}{I}^{\alpha}_{t}g\right)\right) (t)=g(t), \quad \alpha>0, \quad t \in \mathbb{\hat{J}}.
	\end{equation}
\end{defn}

The relationship between  Riemann-Liouville and Caputo fractional derivatives are as follows \cite{I. Podlubny}:
\begin{equation} \label{relationship}
\left( \prescript{C}{t_{0}}D^{\alpha}_{t}g\right) (t)=\left( \prescript{RL}{t_{0}}D^{\alpha}_{t}g\right) (t)-\sum_{k=0}^{n-1}\frac{(t-t_{0})^{k-\alpha}f^{(k)}(t_{0})}{\Gamma(k-\alpha+1)}, \quad n-1<\alpha\leq n, \quad n \in \mathbb{N}.
\end{equation}

The following results are useful in solving fractional differential equations.
\begin{defn}[\cite{Sneddon,Whittaker-Watson}]
A function $g$ is said to be \textit{exponentially bounded} on $[0,\infty)$ if it satisfies an inequality of the form:

\begin{equation*}
|g(t)| \leq M e^{\sigma t}, \quad t\geq T,
\end{equation*}
for some real constants $M>0$, $T>0$ and $\sigma\in \mathbb{R}$.
\end{defn}

\begin{defn}[\cite{Sneddon,Whittaker-Watson}] 
If $g:[0,\infty)\to\mathbb{R}$ is exponentially bounded for $t\geq0$, then the Laplace integral transform $\mathscr{L}\left\lbrace g(t)\right\rbrace (s)$ defined by
\end{defn}
\begin{equation*}
G(s)=\mathscr{L}\left\lbrace g(t)\right\rbrace(s) =\int\limits_{0}^{\infty}e^{-st}g(t)\mathrm{d}t,
\end{equation*}
exists for $s\in\mathbb{C}$ and is an analytic function of $s$ for $\Re(s)>0$ and Laplace inversion formula is defined as
\begin{equation*}
\mathscr{L}^{-1}\left\lbrace G(s)\right\rbrace (t)\coloneqq\frac{1}{2\pi i}\int\limits_{L}e^{st}G(s)ds,
\end{equation*}
where  $g(t)=\mathscr{L}^{-1}\left\lbrace G(s)\right\rbrace (t)$, $t\geq 0$ and $L$ is a closed contour which enclosing the poles (singularities) of $g$.

\begin{defn} [\cite{I. Podlubny}]
	The Laplace integral transform of Riemann-Liouville fractional derivative of order\\
	$\alpha \in (n-1,n]$, $n \in \mathbb{N}$ is given by \cite{I. Podlubny}:
	\begin{equation}\label{lap-RL}
	\mathscr{L}\left\lbrace \left( \prescript{RL}{0}D^{\alpha}_{t}y\right)(t)\right\rbrace (s) =s^{\alpha}Y(s)-\sum_{k=1}^{n}s^{k-1} \left( \prescript{RL}{0}D^{\alpha-k}_{t} y\right)(0),
	\end{equation}
	where $Y(s)$ represents the Laplace transform of the function $y(t)$ . 
\end{defn}
\begin{remark}[\cite{I. Podlubny}\label{remLap-1}] In the special cases, the Laplace integral transform of the Riemann-Liouville fractional differentiation is:\\
	
	\begin{itemize}
		\item If $\alpha \in (0,1]$, then 
		\begin{equation*}
		\mathscr{L} \left\lbrace  \left(\prescript{RL}{0}D^{\alpha}_{t}y\right)(t)\right\rbrace (s)=s^{\alpha}Y(s)-\left( \prescript{RL}{0}D^{\alpha-1}_{t}y\right)(0).
		\end{equation*}
		\item  If $\alpha \in (1,2]$, then
		\begin{equation*}
		\mathscr{L}\left\lbrace \left( \prescript{RL}{0}D^{\alpha}_{t}y\right)(t)\right\rbrace (s)=s^{\alpha}Y(s)- \left( \prescript{RL}{0}D^{\alpha-1}_{t}y\right)(0)-s \left( \prescript{RL}{0}D^{\alpha-2}_{t}y\right)(0).
		\end{equation*}
	\end{itemize}
\end{remark}

\begin{defn} 
	The Laplace integral transform of Caputo fractional derivative of order $\alpha \in (n-1,n]$, $n \in \mathbb{N}$ is given by \cite{I. Podlubny}:
	\begin{equation}\label{lap-C}
	\mathscr{L}\left\lbrace \left( \prescript{C}{0}D^{\alpha}_{t}y\right)(t)\right\rbrace (s) =s^{\alpha}Y(s)-\sum_{k=0}^{n-1}s^{\alpha-k-1}\prescript{}{}y^{(k)}(0),
	\end{equation}
	where $Y(s)$ represents the Laplace transform of the function $y(t)$ . 
\end{defn}
\begin{remark}\label{remLap-2} In the special cases, the Laplace integral transform of the Caputo fractional differentiation is:\\
	
	\begin{itemize}
		\item If $\alpha \in (0,1]$, then 
		\begin{equation*}
		\mathscr{L} \left\lbrace  \left(\prescript{C}{0}D^{\alpha}_{t}y\right)(t)\right\rbrace (s)=s^{\alpha}Y(s)-s^{\alpha-1}y_{0}, \quad \text{where}\quad y_{0}=y(0).
		\end{equation*}
		\item  If $\alpha \in (1,2]$, then
		\begin{equation*}
		\mathscr{L}\left\lbrace \left( \prescript{C}{0}D^{\alpha}_{t}y\right)(t)\right\rbrace (s)=s^{\alpha}Y(s)-s^{\alpha-1}y_{0}-s^{\alpha-2}y^{\prime}_{0}, \quad \text{where} \quad y_{0}=y(0) \quad \text{and} \quad y^{\prime}_{0}=y^{\prime}(0).
		\end{equation*}
	\end{itemize}
\end{remark}
\begin{defn}[\cite{Sneddon,Whittaker-Watson}]
Let $f$ and $g$ be both piece-wise continuous functions on $[0,\infty)$. 
Then the integral in
\begin{equation*}
f\ast g \coloneqq(f\ast g)(t)=\int\limits_0^{t}f(t-s) g(s)ds,
\end{equation*}
is called the convolution operator of two functions $f$ and $g$ which is well-defined and finite for any $t\geq 0$ and it has the commutativity property:
\begin{equation*}
f\ast g=g\ast f.
\end{equation*}
\end{defn}
\begin{thm} [\cite{Sneddon,Whittaker-Watson}]
Suppose that $f$ and $g$ are piece-wise continuous and exponentially bounded functions on $[0,\infty)$.
Then the Laplace transform of convolution operator of two functions $f$ and $g$, given on $[0,\infty)$, has the following property :
\begin{equation*}
\mathscr{L}\left\lbrace \left( f\ast g\right) (t)\right\rbrace(s)=\mathscr{L} \left\lbrace f(t)\right\rbrace(s) \mathscr{L}\left\lbrace g(t)\right\rbrace (s), \quad s\in\mathbb{C}.
\end{equation*}
\end{thm}

The Mittag-Leffler function is a generalization of the exponential function, first proposed in 1903 \cite{1903} as a single-parameter function of one variable, defined using a convergent infinite series. Extensions to two, three and multi-parameters are well known and thoroughly studied in textbooks such as \cite{Gorenflo,Kiryakova V.} which are involving single power series in one variable \cite{Kiryakova V. S.,Virginia K.,K. Virginia}. Extensions to two, three, or more variables, involving correspondingly double, triple, or multiple power series, have been studied more recently \cite{H-A-F-M, Fernandez-Kurt-Ozarslan, H-A-O-M, Saxena-Kalla-Saxena}.

\begin{defn}[\cite{1903}] 
	The classical Mittag-Leffler function is defined by
	\begin{equation*}
	E_{\alpha}(t)= \sum_{i=0}^{\infty}\frac{t^{i}}{\Gamma(i \alpha +1)}, \quad  \alpha>0, t\in\mathbb{R}.
	\end{equation*}
	\begin{remark}[\cite{Diethelm}]
		The Mittag-Leffler functions are often used in a form where the variable inside the brackets is not $t$  but a fractional power $t^{\alpha}$, or even a constant multiple $\lambda t^{\alpha}$, as follows:
		\begin{equation*}
		E_{\alpha}(\lambda t^{\alpha})= \sum_{i=0}^{\infty}\frac{\lambda^i t^{i\alpha}}{\Gamma(i \alpha +1)}, \quad  \alpha >0, t, \lambda \in\mathbb{R}.
		\end{equation*}
	\end{remark}

	The two-parameter Mittag-Leffler function  \cite{Gorenflo} is given by
	\begin{equation*}
	E_{\alpha,\beta}(t)= \sum_{i=0}^{\infty}\frac{t^{i}}{\Gamma(i \alpha +\beta)}, \quad  \alpha>0, \beta\in\mathbb{R}, t\in\mathbb{R}.
	\end{equation*}
	
	The $l$-th derivative of two-parameter Mittag-Leffler function \cite{Gorenflo} is defined by
	\begin{equation*}
   \frac{d^{l}}{dt^{l}}E_{\alpha,\beta}(t)=E^{(l)}_{\alpha,\beta}(t)= \sum_{i=0}^{\infty}\frac{(i+l)!}{i!}\frac{t^{i}}{\Gamma(i \alpha+l\alpha +\beta)}, \quad l\in\mathbb{N},  \alpha>0, \beta\in\mathbb{R}, t\in\mathbb{R}.
\end{equation*}

	The three-parameter Mittag-Leffler function \cite{Prabhakar} is determined by
	\begin{equation*}
	E_{\alpha,\beta}^{\gamma}(t)= \sum_{i=0}^{\infty}\frac{(\gamma)_i}{\Gamma(i \alpha +\beta)}\frac{t^{i}}{i!}, \quad  \alpha >0,\beta,\gamma\in\mathbb{R},  t\in\mathbb{R},
	\end{equation*}
	where $(\gamma)_i$ is the Pochhammer symbol denoting $\frac{\Gamma(\gamma+i)}{\Gamma(\gamma)}$. These series are convergent, locally uniformly in $\tau$ , provided the $\alpha>0$ condition is satisfied. Note that
	\[
	E_{\alpha,\beta}^1(t)=E_{\alpha,\beta}(t),\quad E_{\alpha,1}(t)=E_{\alpha}(t),\quad E_1(t)=\exp(t).
	\]
\end{defn}

The next lemma includes Laplace integral transform of three-parameter Mittag-Leffler function which will be used throughout the proof of Lemma \ref{lem:2}.
\begin{lem} \label{lem:1}
	For $\alpha > \beta >0$, $\lambda \in \mathbb{R}$, $l \in \mathbb{N}_{0}=\left\lbrace 0,1,2, \dots \right\rbrace $ and $\Re(s)>0$, we have:
	\begin{align*}
	\mathscr{L}^{-1}\Bigl\{ \frac{1}{(s^{\alpha}-\lambda s^{\beta})^{l+1}}\Bigr\}(\tau)&= t^{(l+1)\alpha-1}\sum_{k=0}^{\infty}\binom{l+k}{k}\frac{\lambda^{k}t^{k(\alpha-\beta)}}{\Gamma(k(\alpha-\beta)+(l+1)\alpha)} \\
	&\coloneqq t^{(l+1)\alpha-1}E^{l+1}_{\alpha-\beta, (l+1)\alpha}(\lambda t^{\alpha-\beta}).
	\end{align*}
\end{lem}

\begin{proof}
	By using the Taylor series representation of $\frac{1}{(1-t)^{l+1}}, l \in \mathbb{N}_{0}$ of the form
	\begin{equation*}
	\frac{1}{(1-t)^{l+1}}= \sum_{k=0}^{\infty}\binom{l+k}{k}t^{k}, \quad |t|<1,
	\end{equation*}
	we achieve that
	\begin{align*}
	\frac{1}{(s^{\alpha}-\lambda s^{\beta})^{l+1}}=\frac{1}{(s^{\alpha})^{l+1}}\frac{1}{(1-\frac{\lambda}{s^{\alpha-\beta}} )^{l+1}}&=\frac{1}{s^{(l+1)\alpha}}\sum_{k=0}^{\infty}\binom{l+k}{k}\Big(\frac{\lambda}{s^{\alpha-\beta}}\Big)^{k}\\
	&=\sum_{k=0}^{\infty}\binom{l+k}{k}\frac{\lambda^{k}}{s^{k(\alpha-\beta)+(l+1)\alpha}}.
	\end{align*}
	Taking inverse Laplace transform of the above function, we get the desired result:
	\begin{align*}
	\mathscr{L}^{-1}\Bigl\{ \frac{1}{(s^{\alpha}-\lambda s^{\beta})^{l+1}}\Bigr\}(t)&=\sum_{k=0}^{\infty}\lambda^{k}\binom{l+k}{k}\mathscr{L}^{-1}\Bigr\{\frac{1}{s^{k(\alpha-\beta)+(l+1)\alpha}}\Bigr\}(t)\\
	&=\sum_{k=0}^{\infty}\lambda^{k}\binom{l+k}{k}\frac{t^{k(\alpha-\beta)+(l+1)\alpha-1}}{\Gamma(k(\alpha-\beta)+(l+1)\alpha)}\\
	&=t^{(l+1)\alpha-1}E^{l+1}_{\alpha-\beta, (l+1)\alpha}(\lambda t^{\alpha-\beta}),
	\end{align*}
	 which is the required result. We have required an extra condition on $s$ for convergence of the binomial type series in the Laplace domain, namely that
	\begin{align*}
	s^{\alpha-\beta}>|\lambda|.
	 \end{align*}
	 However, this condition can be removed at the end, by analytic continuation of both sides of the identity, to give the desired result for all $s \in\mathbb{C}$ satisfying $\Re(s)>0$. The proof is complete.
 \end{proof}
\begin{defn}[\cite{Fernandez-Kurt-Ozarslan}] \label{Def:bML}
	We consider the bivariate Mittag-Leffler function defined by
	\begin{equation}
	\label{bML4}
	E_{\alpha,\beta,\gamma}^{\delta}(u,v)=\sum_{l=0}^{\infty}\sum_{k=0}^{\infty}\frac{(\delta)_{l+k}}{\Gamma(l\alpha+k\beta+\gamma)}\frac{u^lv^k}{l!k!},\quad \alpha,\beta>0,\gamma,\delta\in\mathbb{R},u,v\in\mathbb{R}.
	\end{equation}
\end{defn}

If we write $u=\lambda t^{\alpha}$ and $v=\mu t^{\beta}$ for a single variable $t$, and multiply by a power function $t^{\gamma-1}$, we derive the following univariate version:
\begin{equation}\label{bivtype}
t^{\gamma-1}	E_{\alpha,\beta,\gamma}^{\delta}(\lambda t^{\alpha},\mu t^{\beta})=\sum_{l=0}^{\infty}\sum_{k=0}^{\infty}\frac{(\delta)_{l+k}}{\Gamma(l\alpha+k\beta+\gamma)}\frac{\lambda^l\mu^k}{l!k!}t^{l\alpha+k\beta+\gamma-1}.
\end{equation}
Note that when $\delta=1$,
\begin{align*}
E_{\alpha,\beta,\gamma}^{1}(\lambda t^{\alpha},\mu t^{\beta})&=\sum_{l=0}^{\infty}\sum_{k=0}^{\infty}\frac{(1)_{l+k}}{\Gamma(l\alpha+k\beta+\gamma)}\frac{\lambda^l\mu^k}{l!k!}t^{l\alpha+k\beta+\gamma-1}\\
&=\sum_{l=0}^{\infty}\sum_{k=0}^{\infty}\frac{(l+k)!}{l!k!}\frac{\lambda^l\mu^k}{\Gamma(l\alpha+k\beta+\gamma)}t^{l\alpha+k\beta+\gamma-1}\\
&=\sum_{l=0}^{\infty}\sum_{k=0}^{\infty}\binom{l+k}{k}\frac{\lambda^l\mu^k}{\Gamma(l\alpha+k\beta+\gamma)}t^{l\alpha+k\beta+\gamma-1}.
\end{align*}
For simplicity, we denote $E_{\alpha,\beta,\gamma}^{1}(\lambda t^{\alpha},\mu t^{\beta})\coloneqq E_{\alpha,\beta,\gamma}(\lambda t^{\alpha},\mu t^{\beta})$ in our results for this paper.
\begin{lem} \label{lem:2}
	For $\alpha > \beta$, $\alpha >\gamma$, $\lambda, \mu \in \mathbb{R}$ and $\Re(s)>0$, the following result holds true:
	\begin{align*}
	\mathscr{L}^{-1}\Bigl\{\frac{s^{\gamma}}{s^{\alpha}-\mu s^{\beta}-\lambda} \Bigr\}(t)&= t^{\alpha-\gamma-1}\sum_{l=0}^{\infty} \sum_{k=0}^{\infty}\binom{l+k}{k}\frac{\lambda^{l}\mu^{k}t^{l\alpha+k(\alpha-\beta)}}{\Gamma(l\alpha+k(\alpha-\beta)+\alpha-\gamma)} \\
	&=t^{\alpha-\gamma-1}E_{\alpha,\alpha-\beta, \alpha-\gamma}(\lambda t^{\alpha}, \mu t^{\alpha-\beta}).
	\end{align*}
\end{lem}
\begin{proof}
	$\frac{s^{\gamma}}{s^{\alpha}-\mu s^{\beta}-\lambda}$ can be written via a series expansion as follows:
	\begin{equation*}
	\frac{s^{\gamma}}{s^{\alpha}-\mu s^{\beta}-\lambda}=\frac{s^{\gamma}}{s^{\alpha}-\mu s^{\beta}}\frac{1}{1-\frac{\lambda}{s^{\alpha}-\mu s^{\beta}}}=\sum_{l=0}^{\infty}\frac{\lambda^{l}s^{\gamma}}{(s^{\alpha}-\mu s^{\beta})^{l+1}}.
	\end{equation*}
	Then applying Lemma \ref{lem:1} to the last expression, we acquire that
	\begin{align*}
	\frac{s^{\gamma}}{s^{\alpha}-\mu s^{\beta}-\lambda}&=\sum_{l=0}^{\infty}\frac{\lambda^{l}s^{\gamma}}{s^{(l+1)\alpha}}\frac{1}{(1-\frac{\mu}{s^{\alpha-\beta}})^{l+1}}\\
	&=\sum_{l=0}^{\infty}\frac{\lambda^{l}s^{\gamma}}{s^{(l+1)\alpha}}\sum_{k=0}^{\infty}\binom{l+k}{k}\Big(\frac{\mu}{s^{\alpha-\beta}}\Big)^{k}\\
	&=\sum_{l=0}^{\infty}\sum_{k=0}^{\infty}\binom{l+k}{k}\frac{\lambda^{l}\mu^{k}}{s^{(l+1)\alpha+k(\alpha-\beta)-\gamma}}.
	\end{align*}
	Taking inverse Laplace transform of the aforementioned function, we attain:
	\begin{align*}
	\mathscr{L}^{-1}\Bigl\{\frac{s^{\gamma}}{s^{\alpha}-\mu s^{\beta}-\lambda} \Bigr\}(t)&= \sum_{l=0}^{\infty} \sum_{k=0}^{\infty}\binom{l+k}{k}\mu^{l}\lambda^{k}\mathscr{L}^{-1}\Bigl\{\frac{1}{s^{(l+1)\alpha+k(\alpha-\beta)-\gamma}} \Bigr\}(t)\\
	&=t^{\alpha-\gamma-1}\sum_{l=0}^{\infty} \sum_{k=0}^{\infty}\binom{l+k}{k}\frac{\lambda^{l}\mu^{k}t^{l\alpha+k(\alpha-\beta)}}{\Gamma(l\alpha+k(\alpha-\beta)+\alpha-\gamma)}\\
	&=t^{\alpha-\gamma-1}E_{\alpha,\alpha-\beta, \alpha-\gamma}(\lambda t^{\alpha},\mu t^{\alpha-\beta}), 
	\end{align*}
	which is the desired result. We have required extra conditions on $s$ for convergence of the binomial type series in the Laplace domain, namely that
	\begin{align*}
	  &s^{\alpha-\beta}>|\mu|,\\
	  &|s^{\alpha}-\mu s^{\beta}| > |\lambda|.
	\end{align*}
	However, these conditions can be removed at the end, by analytic continuation of both sides of the identity, to give the desired result for all $s \in\mathbb{C}$ satisfying $\mathrm{\Re}(s)>0$. The proof is complete.
\end{proof}

\begin{lem} \label{Lem:biML}
	For any parameters $\alpha,\beta,\gamma,\lambda,\mu\in\mathbb{R}$ satisfying $\alpha,\beta>0$ and $\gamma-1>\lfloor\alpha\rfloor$, we have
	\begin{equation}\label{N1}
	\prescript{C}{0}D^{\alpha}_{t}\Big[t^{\gamma-1}E_{\alpha,\beta,\gamma}(\lambda t^{\alpha},\mu t^{\beta})\Big]=t^{\gamma-\alpha-1}E_{\alpha,\beta,\gamma-\alpha}(\lambda t^{\alpha},\mu t^{\beta}), \quad t>0.
\end{equation}
\end{lem}
\begin{proof}
	We have the following formula for Caputo derivatives of power functions \cite{I. Podlubny,Kilbas-Srivastava-Trujillo}:
	\begin{equation} \label{Cap}
	\prescript{C}{0}D^{\nu}_{t}\left(\frac{t^{\eta}}{\Gamma(\eta+1)}\right)= 	\begin{cases}
	\frac{t^{\eta-\nu}}{\Gamma(\eta-\nu+1)},\qquad\eta>\lfloor\nu\rfloor,\\
	\qquad 0, \quad \qquad \eta=0,1,2,\dots,\lfloor\nu\rfloor,\\
	\text{undefined}, \qquad \text{otherwise}.
	\end{cases}
	\end{equation}
	Therefore, the given condition $\gamma-1>\lfloor\alpha\rfloor$, from \eqref{Cap} we can attain
	\begin{align*}
	\prescript{C}{0}D^{\alpha}_{t}\left[ t^{\gamma-1}E_{\alpha,\beta,\gamma}(\lambda t^{\alpha},\mu t^{\beta})\right]  &=\prescript{C}{0}D^{\alpha}_{t}\left[\sum_{l=0}^{\infty}\sum_{k=0}^{\infty}\binom{l+k}{k}\frac{\lambda^l\mu^k t^{l\alpha+k\beta+\gamma-1}}{\Gamma(l\alpha+k\beta+\gamma)}\right]\\
	&=\sum_{l=0}^{\infty}\sum_{k=0}^{\infty}\binom{l+k}{k}\lambda^l\mu^k\prescript{C}{0}D^{\alpha}_{t}\left( \frac{t^{l\alpha+k\beta+\gamma-1}}{\Gamma(l\alpha+k\beta+\gamma)} \right) \\ \nonumber
	&=\sum_{l=0}^{\infty}\sum_{k=0}^{\infty}\binom{l+k}{k}\frac{\lambda^l\mu^kt^{l\alpha+k\beta+\gamma-\alpha-1}}{\Gamma(l\alpha+k\beta+\gamma-\alpha)}\\ \nonumber
	&=t^{\gamma-\alpha-1}E_{\alpha,\beta,\gamma-\alpha}(\lambda t^{\alpha},\mu t^{\beta}), \quad t>0.
	\end{align*}
	The proof is complete.
\end{proof}

\begin{lem} \label{Lem:biML-RL}
	For any parameters $\alpha,\beta,\gamma,\lambda,\mu\in\mathbb{R}$ satisfying $\alpha,\beta,\gamma>0$, we have
	\begin{equation}\label{RL}
	\prescript{RL}{0}D^{\alpha}_{t}\Big[t^{\gamma-1}E_{\alpha,\beta,\gamma}(\lambda t^{\alpha},\mu t^{\beta})\Big]=t^{\gamma-\alpha-1}E_{\alpha,\beta,\gamma-\alpha}(\lambda t^{\alpha},\mu t^{\beta}), \quad t>0.
	\end{equation}
\end{lem}
\begin{proof}
	We have the following formula for Riemann-Liouville derivatives of power functions \cite{I. Podlubny,Kilbas-Srivastava-Trujillo}:
	\begin{equation} \label{powerRL}
	\prescript{RL}{0}D^{\nu}_{t}\left(\frac{t^{\eta}}{\Gamma(\eta+1)}\right)= 	
	\frac{t^{\eta-\nu}}{\Gamma(\eta-\nu+1)},\quad \nu, \eta \in \mathbb{R}, \quad \eta>-1.
	\end{equation}
	Therefore, given the condition $\gamma>0$, in accordance with \eqref{powerRL} we will get the same result with \eqref{N1}.
	The proof is complete.
\end{proof}

\begin{defn}
	Let $\lambda_{i},\mu_{j} \in \mathbb{R}$, $\alpha_{i},\beta_{j} \in \mathbb{R}$, $i=1,2,\ldots,p$, $j=1,2,\ldots,q$. Generalized Wright function or Fox-Wright function $\prescript{}{p}{\Psi_{q}}(\cdot):\mathbb{R}\to \mathbb{R}$ is defined by
	\begin{equation}\label{fox}
	\prescript{}{p}{\Psi_{q}}(t)=\prescript{}{p}{\Psi_{q}}\left[\begin{array}{ccc}
	(\lambda_{i},\alpha_i)_{1,p} \\
	(\mu_j,\beta_j)_{1,q}
	\end{array}\Big|t\right]=\sum_{k=0}^{\infty}\frac{\prod\limits_{i=1}^{p}\Gamma(\lambda_{i}+\alpha_{i}k)}{\prod\limits_{j=1}^{q}\Gamma(\mu_{j}+\beta_{j}k)}\frac{t^{k}}{k!}.
	\end{equation}
\end{defn} 

The Fox-Wright function was established by Fox \cite{Fox} and Wright \cite{Wright}. If the following condition holds
\begin{equation*}
\sum_{j=1}^{q}\beta_j-\sum_{i=1}^{p}\alpha_{i}>-1,
\end{equation*}
then the series in \eqref{fox} is convergent for arbitrary $t\in \mathbb{R}$.

\section{Leibniz integral rule}\label{generalized Leibniz rule}

In this section, we formulate  Leibniz integral rule for higher order derivatives on Lebesgue integration. It is known that according to the suitable conditions, we can differentiate under the integral sign for Lebesgue integrals \cite{Chen}. We begin with the first derivative of a Lebesgue integral on $X\subseteq\mathbb{R}$.

\begin{thm}[\label{Lebesgue}\cite{Chen}]
	Assume that $X,Y \subseteq \mathbb{R}$ are intervals. Suppose also that the function $f : X \times Y\to\mathbb{R}$ satisfies
	the following assumptions:
	
	(a) For every fixed $y\in Y$, the function $f(\cdot, y)$ is measurable on $X$;
	
	(b) The partial derivative $\frac{\partial}{\partial y}f(x, y)$ exists for every interior point $(x, y) \in X \times Y$;
	
	(c) There exists a non-negative integrable function $g$ such that  $\left|\frac{\partial}{\partial y}f(x, y)\right| \leq g(x)$ exists for every interior point $(x, y) \in X \times Y$;
	
	(d) There exists $y_{0}\in Y$ such that $f(x, y_{0})$ is integrable on $X$.
	
	Then for every $y \in Y$, the Lebesgue integral
	\begin{equation*}
	\int\limits_{X}f(x, y)\mathrm{d}x
	\end{equation*}
	exists. Furthermore, the function $F : Y \to \mathbb{R}$, defined by
	\begin{equation*}
	F(y)=\int\limits_{X}f(x, y)\mathrm{d}x
	\end{equation*}
	for every $y \in Y,$ is differentiable at every interior point of $ Y$, and the derivative of $F(y)$ satisfies
	\begin{equation*}
	F^{\prime}(y)=\int\limits_{X}\frac{\partial}{\partial y}f(x, y)\mathrm{d}x.
	\end{equation*}
\end{thm}
The well-known rule for the differentiation of an integral depending on
a parameter with the upper limit also depends on the same parameter,
namely:

\begin{coroll}[\cite{Fikhtengoltz}]
	If $X=(y_{0},y)$ and assumptions of Theorem \ref{Lebesgue} are fulfilled, then the following relation holds true for all $y \in Y$:
	\begin{equation}
	\frac{d}{dy}\int\limits_{y_{0}}^{y}f(x,y)\mathrm{d}x=\int\limits_{y_{0}}^{y}\frac{\partial}{\partial y} f(x,y)\mathrm{d}x + \lim\limits_{x \to y-0}f(x,y),\quad y \in X.
	\end{equation}
\end{coroll}

So, the formula of differentiation under the integral sign for $K(t,s)$ with respect to $t$ is
\begin{equation}\label{first der}
\frac{d}{dt}\int\limits_{t_{0}}^{t}K(t,s)\mathrm{d}s=\int\limits_{t_{0}}^{t}\frac{\partial}{\partial t} K(t,s)\mathrm{d}s + \lim\limits_{s \to t-0}K(t,s), \quad t \in \mathbb{\hat{J}}.
\end{equation}

Using the formula \eqref{first der}, we define the second-order derivative of the integral depending on $t$:
\allowdisplaybreaks
\begin{align*}
\frac{d^{2}}{dt^{2}}\int\limits_{t_{0}}^{t}K(t,s)\mathrm{d}s&=
\frac{d}{dt}\left(\frac{d}{dt} \int\limits_{t_{0}}^{t}K(t,s)\mathrm{d}s\right)=
\frac{d}{dt}\left(\lim\limits_{s\to t-0}K(t,s)+\int\limits_{t_{0}}^{t}\frac{\partial}{\partial t}K(t,s)\mathrm{d}s\right) \\
&=\frac{d}{dt}\lim\limits_{s\to t-0}K(t,s)+\frac{d}{dt}\int\limits_{t_{0}}^{t}\frac{\partial}{\partial t}K(t,s)\mathrm{d}s  \\
&=\frac{d}{dt}\lim\limits_{s\to t-0}K(t,s)+\lim\limits_{s\to t-0}\frac{\partial}{\partial t}K(t,s)+\int\limits_{t_{0}}^{t}\frac{\partial^2}{\partial t^2}K(t,s)\mathrm{d}s\\ &=\sum_{l=1}^{2}\frac{d^{l-1}}{dt^{l-1}}\lim\limits_{s\to t-{0}}\frac{\partial^{2-l}}{\partial t^{2-l}}K(t,s)+\int\limits_{t_{0}}^{t}\frac{\partial^{2}}{\partial t^{2}}K(t,s)\mathrm{d}s, \quad t \in \mathbb{\hat{J}}.
\end{align*}
Then, the third-order differentiation of the integral will be:
\allowdisplaybreaks	
\begin{align*}
&\frac{d^{3}}{dt^{3}}\int\limits_{t_{0}}^{t}K(t,s)\mathrm{d}s=
\frac{d}{dt}\left(\frac{d^2}{dt^2} \int\limits_{t_{0}}^{t}K(t,s)\mathrm{d}s\right)\\
=&\frac{d}{dt}\left(\frac{d}{dt}\lim\limits_{s\to t-0}K(t,s)+\lim\limits_{s\to t-0}\frac{\partial}{\partial t}K(t,s)+\int\limits_{t_{0}}^{t}\frac{\partial^2}{\partial t^2}K(t,s)\mathrm{d}s\right) \\
=&\frac{d^{2}}{dt^{2}}\lim\limits_{s\to t-0}K(t,s)+\frac{d}{dt}\lim\limits_{s\to t-0}\frac{\partial}{\partial t}K(t,s)+\frac{d}{dt}\int\limits_{t_{0}}^{t}\frac{\partial^2}{\partial t^2}K(t,s)\mathrm{d}s\\
=&\frac{d^{2}}{dt^{2}}\lim\limits_{s\to t-0}K(t,s)+\frac{d}{dt}\lim\limits_{s\to t-0}\frac{\partial}{\partial t}K(t,s)\\
+&\lim\limits_{s\to t-0}\frac{\partial^2}{\partial t^2}K(t,s)+\int\limits_{t_{0}}^{t}\frac{\partial^{3}}{\partial t^{3}}K(t,s)\mathrm{d}s\\
=&\sum_{l=1}^{3}\frac{d^{l-1}}{dt^{l-1}}\lim\limits_{s\to t-{0}}\frac{\partial^{3-l}}{\partial t^{3-l}}K(t,s)+\int\limits_{t_{0}}^{t}\frac{\partial^{3}}{\partial t^{3}}K(t,s)\mathrm{d}s, \quad t \in \mathbb{\hat{J}}.
\end{align*}
 Thus, we establish $n$th derivative of the integral for $n \in \mathbb{N}$ which depends on $t$ by recursively as follows:

\begin{align*}
&\frac{d^{n}}{dt^{n}}\int\limits_{t_{0}}^{t}K(t,s)\mathrm{d}s
=\frac{d}{dt}\left( \frac{d^{n-1}}{dt^{n-1}}\int\limits_{t_{0}}^{t}K(t,s)\mathrm{d}s\right)\\
=&\frac{d^{n-1}}{dt^{n-1}}\lim\limits_{s\to t-0}K(t,s)+\frac{d^{n-2}}{dt^{n-2}}\lim\limits_{s\to t-0}\frac{\partial}{\partial t}K(t,s)+\cdots + \frac{d}{dt}\lim\limits_{s\to t-0}\frac{\partial^{n-2}}{\partial t^{n-2}}K(t,s)\\
+&\lim\limits_{s\to t-0}\frac{\partial^{n-1}}{\partial t^{n-1}}K(t,s)+\int\limits_{t_{0}}^{t}\frac{\partial^{n}}{\partial t^{n}}K(t,s)\mathrm{d}s\\
=&\sum_{l=1}^{n}\frac{d^{l-1}}{dt^{l-1}}\lim\limits_{s\to t-0}\frac{\partial^{n-l}}{\partial t^{n-l}}K(t,s)+\int\limits_{t_{0}}^{t}\frac{\partial^{n}}{\partial t^{n}}K(t,s)\mathrm{d}s, \quad t \in \mathbb{\hat{J}}.
\end{align*}
In the next theorem, we state and prove the  Leibniz rule for higher order derivatives.

\begin{thm}\label{K(t,s)}
Let the function $K:\mathbb{J}\times\mathbb{J}\to\mathbb{R}$ be such that the following assumptions are fulfilled:

	(a) For every fixed $t\in \mathbb{J}$, the function $\frac{\partial^{n-1}}{\partial t^{n-1}}K(t,s)$ is measurable on $\mathbb{J}$ and integrable on $\mathbb{J}$ with respect to for some $t^{*}\in \mathbb{J}$;
	
	(b) The partial derivative $\frac{\partial^{n}}{\partial t^{n}}K(t,s)$ exists for every interior point $(t, s) \in \mathbb{\hat{J}} \times \mathbb{\hat{J}}$;
	
	(c) There exists a non-negative integrable function $g$ such that  $\left|\frac{\partial^{n}}{\partial t^{n}}K(t,s)\right| \leq g(s)$ for every interior point $(t, s) \in \mathbb{\hat{J}}\times \mathbb{\hat{J}}$;
	
	(d) The derivative $\frac{d^{l-1}}{dt^{l-1}}\lim\limits_{s\to t-{0}}\frac{\partial^{n-l}}{\partial t^{n-l}}K(t,s)$, $l=1,2,\ldots,n$ exists for every interior point $(t, s) \in \mathbb{\hat{J}} \times \mathbb{\hat{J}}$.
	
Then, the following relation holds true for $n$th derivative under Lebesgue integration for $n \in \mathbb{N}$:
	
	\begin{equation}\label{eq1}
	\frac{d^{n}}{dt^{n}}\int\limits_{t_{0}}^{t}K(t,s)\mathrm{d}s= \sum_{l=1}^{n}\frac{d^{l-1}}{dt^{l-1}}\lim\limits_{s\to t-{0}}\frac{\partial^{n-l}}{\partial t^{n-l}}K(t,s)+\int\limits_{t_{0}}^{t}\frac{\partial^{n}}{\partial t^{n}}K(t,s)\mathrm{d}s, \quad t \in \mathbb{\hat{J}}.
	\end{equation}
\end{thm}

\begin{proof}
	Using mathematical induction principle, we prove above theorem.
	It is obvious that the equation \eqref{eq1} is true for $n=1$ \cite{I. Podlubny}:
	\begin{equation*}
	\frac{d}{dt}\int\limits_{t_{0}}^{t}K(t,s)\mathrm{d}s=\lim\limits_{s\to t-0}K(t,s)+\int\limits_{t_{0}}^{t}\frac{\partial}{\partial t}K(t,s)\mathrm{d}s.
	\end{equation*}
	We assume that \eqref{eq1} holds true for $n=k$:
	\begin{equation*}
	\frac{d^{k}}{dt^{k}}\int\limits_{t_{0}}^{t}K(t,s)\mathrm{d}s= \sum_{l=1}^{k}\frac{d^{l-1}}{dt^{l-1}}\lim\limits_{s\to t-0}\frac{\partial^{k-l}}{\partial t^{k-l}}K(t,s)+\int\limits_{t_{0}}^{t}\frac{\partial^{k}}{\partial t^{k}}K(t,s)\mathrm{d}s, \quad t \in \mathbb{\hat{J}}.
	\end{equation*}
	We prove that \eqref{eq1} is true for $n=k+1$:
	\allowdisplaybreaks
	\begin{align*}
	\frac{d^{k+1}}{dt^{k+1}}\int\limits_{t_{0}}^{t}K(t,s)\mathrm{d}s&=\frac{d}{dt}\left( \frac{d^{k}}{dt^{k}}\int\limits_{t_{0}}^{t}K(t,s)\mathrm{d}s\right) \\ &=\frac{d}{dt}\left( \sum_{l=1}^{k}\frac{d^{l-1}}{dt^{l-1}}\lim\limits_{s\to t-0}\frac{\partial^{k-l}}{\partial t^{k-l}}K(t,s)+\int\limits_{t_{0}}^{t}\frac{\partial^{k}}{\partial t^{k}}K(t,s)\mathrm{d}s\right)\\
	&=\frac{d}{dt}\Big(\lim\limits_{s\to t-0}\frac{\partial^{k-1}}{\partial t^{k-1}}K(t,s)+ \frac{d}{dt}\lim\limits_{s\to t-0}\frac{\partial^{k-2}}{\partial t^{k-2}}K(t,s)+\cdots +\frac{d^{k-1}}{dt^{k-1}}\lim\limits_{s\to t-0} K(t,s)\Big)\\
	&+\frac{d}{dt}\int\limits_{t_{0}}^{t}\frac{\partial^{k}}{\partial t^{k}}K(t,s)\mathrm{d}s\\
	&=\frac{d}{dt}\lim\limits_{s\to t-0}\frac{\partial^{k-1}}{\partial t^{k-1}}K(t,s)+ \frac{d^{2}}{dt^{2}}\lim\limits_{s\to t-0}\frac{\partial^{k-2}}{\partial t^{k-2}}K(t,s)+\cdots +\frac{d^{k}}{dt^{k}}\lim\limits_{s\to t-0} K(t,s)\\
	&+\lim\limits_{s\to t-0}\frac{\partial^{k}}{\partial t^{k}}K(t,s)+\int\limits_{t_{0}}^{t}\frac{\partial^{k+1}}{\partial t^{k+1}}K(t,s)\mathrm{d}s\\
	&=\sum_{l=2}^{k+1}\frac{d^{l-1}}{dt^{l-1}}\lim\limits_{s \to t-0}\frac{\partial^{k-l+1}}{\partial t^{k-l+1}}K(t,s)+\lim\limits_{s\to t-0}\frac{\partial^{k}}{\partial t^{k}}K(t,s)+\int\limits_{t_{0}}^{t}\frac{\partial^{k+1}}{\partial t^{k+1}}K(t,s)\mathrm{d}s\\
	&=\sum_{l=1}^{k+1}\frac{d^{l-1}}{dt^{l-1}}\lim\limits_{s \to t-0}\frac{\partial^{k-l+1}}{\partial t^{k-l+1}}K(t,s)+\int\limits_{t_{0}}^{t}\frac{\partial^{k+1}}{\partial t^{k+1}}K(t,s)\mathrm{d}s, \quad t\in \mathbb{\hat{J}}.
	\end{align*}
	Therefore, the formula \eqref{eq1} holds true for all $n \in \mathbb{N}$ and $t\in \mathbb{\hat{J}}$.
\end{proof}
The following important particular case must be defined for convolution operator of the functions $f$ and $g$.
\begin{coroll}If $K(t,s)=f(t-s)g(s)$ and $t_{0}=0$, and assumptions of Theorem \ref{K(t,s)} are satisfied, then the following relation is true for any $n\in \mathbb{N}$:
	\begin{align}\label{classconvol}
	\frac{d^{n}}{dt^{n}}\int\limits_{0}^{t}f(t-s)g(s)\mathrm{d}s &=\sum_{l=1}^{n}\lim\limits_{s \to t-0}\frac{\partial^{n-l}}{\partial t^{n-l}}f(t-s)\frac{d^{l-1}}{dt^{l-1}}\lim\limits_{s \to t-0}g(s)\nonumber\\
	&+\int\limits_{0}^{t}\frac{\partial^{n}}{\partial t^{n}}f(t-s)g(s)\mathrm{d}s, \quad t >0.
	\end{align}
\end{coroll}
\begin{proof}
	If we write $f(t-s)g(s)$ instead of $K(t,s)$ in \eqref{eq1}, then we obtain
	\begin{align*}
	\frac{d^{n}}{dt^{n}}\int\limits_{0}^{t}f(t-s)g(s)\mathrm{d}s&=\sum_{l=1}^{n}\frac{d^{l-1}}{dt^{l-1}}\left(\lim\limits_{s \to t-0}\left( \frac{\partial^{n-l}}{\partial t^{n-l}}f(t-s)g(s)\right)  \right) \\
	&+\int\limits_{0}^{t}\frac{\partial^{n}}{\partial t^{n}}f(t-s)g(s)\mathrm{d}s\\
	&=\sum_{l=1}^{n}\frac{d^{l-1}}{dt^{l-1}}\left(\lim\limits_{s \to t-0}\frac{\partial^{n-l}}{\partial t^{n-l}}f(t-s)\lim\limits_{s\to t-0}g(s)\right) \\
	&+\int\limits_{0}^{t}\frac{\partial^{n}}{\partial t^{n}}f(t-s)g(s)\mathrm{d}s\\
	&=\sum_{l=1}^{n}\lim\limits_{s \to t-0}\frac{\partial^{n-l}}{\partial t^{n-l}}f(t-s)\frac{d^{l-1}}{dt^{l-1}}\lim\limits_{s \to t-0}g(s)\\
	&+\int\limits_{0}^{t}\frac{\partial^{n}}{\partial t^{n}}f(t-s)g(s)\mathrm{d}s, \quad t>0.
	\end{align*}
	Thus, the proof is complete.
\end{proof}

\section{ Fractional Leibniz integral rules}\label{Generealized fractional Leibniz rules}
Now, we are starting to prove fractional Leibniz integral rule for Riemann-Liouville fractional derivative of order $\alpha \in (n-1,n], n\in \mathbb{N}$. For this, firstly, let us consider partial Riemann-Liouville fractional differentiation operator of order $n-1<\alpha \leq n$, $n\in\mathbb{N}$ with respect to $t$ of a function $K(t,s)$ of two variables $(t,s)\in\mathbb{J}\times\mathbb{J}$, $K:\mathbb{J}\times\mathbb{J}\to\mathbb{R}$, defined by
\cite{Miller-Ross,Oldham-Spanier,Samko-Kilbas-Marichev}:

	\begin{equation}
	\prescript{RL, t}{t_{0}}{D}^{\alpha}_{t} K(t,s) =\frac{\mathrm{\partial}^n}{\mathrm{\partial }t^n}\prescript{t}{t_{0}}{I}^{n-\alpha}_{t}K(t,s) =\frac{1}{\Gamma(n-\alpha)}\frac{\mathrm{\partial }^n}{\mathrm{\partial }t^n}\int\limits^{t}_{t_{0}}(t-s)^{n-\alpha-1}K(s,\tau)\mathrm{d}s, \quad t\in \mathbb{\hat{J}},
	\end{equation}

where $\prescript{t}{t_{0}}{I}^{\alpha}_{t}$ is the partial Riemann-Liouville integral operator of order $\alpha>0$ which is given by:
\begin{equation*}
\prescript{t}{t_{0}}{I}^{\alpha}_{t}K(t,s) =\frac{1}{\Gamma(\alpha)}\int\limits_{t_{0}}^{t}(t-s)^{\alpha-1}K(s,\tau)\mathrm{d}s \\,\quad \text{for} \quad t\in \mathbb{\hat{J}}.
\end{equation*}

 The following important result in the theory of fractional calculus was first proposed by Podlubny \cite{I. Podlubny} for $\alpha \in (0,1]$ in Riemann-Liouville sense as follows:
\begin{equation*}
\prescript{RL}{t_{0}}{D^{\alpha}_{t}}\int\limits_{t_{0}}^{t}K(t,s)\mathrm{d}s=\lim\limits_{s\to t-0}\prescript{t}{s}{I^{1-\alpha}_{t}}K(t,s)+\int\limits_{t_{0}}^{t}\prescript{RL,t}{s}{D^{\alpha}_{t}}K(t,s)\mathrm{d}s, \quad t > t_{0}.
\end{equation*}
Now, we are going to state and prove the following theorem for more general case where $\alpha \in (n-1,n], n \in \mathbb{N}$ which is more useful tool for the testing particular solution of inhomogeneous linear multi-order differential equations with variable coefficients. Note that Matychyn has proposed \cite{I. Matychyn} Leibniz integral rule for Riemann-Liouville derivative of order $0<\alpha\leq 1$ on Lebesgue integration. 
\begin{thm}\label{thm-class}
Let the function $K:\mathbb{J}\times\mathbb{J}\to\mathbb{R}$ be such that the following assumptions are fulfilled:
	
(a) For every fixed $t\in \mathbb{J}$, the function $\hat{K}(t,s)=\prescript{RL,t}{s}{D^{\alpha-1}_{t}}K(t,s)$ is measurable on $\mathbb{J}$ and integrable on $\mathbb{J}$ with respect to some $t^{*}\in \mathbb{J}$;
	
(b) The partial derivative $\prescript{RL,t}{s}{D^{\alpha}_{t}}K(t,s)$ exists for every interior point $(t, s) \in \mathbb{\hat{J}} \times \mathbb{\hat{J}}$;
	
(c) There exists a non-negative integrable function $g$ such that  $\left|\prescript{RL,t}{s}{D^{\alpha}_{t}}K(t,s)\right| \leq g(s)$ for every interior point $(t, s) \in \mathbb{\hat{J}}\times \mathbb{\hat{J}}$;

(d) The derivative $\frac{d^{l-1}}{dt^{l-1}}\lim\limits_{s\to t-0}\prescript{RL,t}{s}{D^{\alpha-l}_{t}}K(t,s)$, $l=1,2,\ldots,n$ exists for every interior point $(t, s) \in \mathbb{\hat{J}} \times \mathbb{\hat{J}}$;
	
Then, the following relation holds true for fractional derivative in Riemann-Liouville sense under Lebesgue integration:

\begin{equation}\label{thm-RL}
\prescript{RL}{t_{0}}{D^{\alpha}_{t}}\int\limits_{t_{0}}^{t}K(t,s)\mathrm{d}s=\sum_{l=1}^{n}\frac{d^{l-1}}{dt^{l-1}}\lim\limits_{s\to t-0}\prescript{RL,t}{s}{D^{\alpha-l}_{t}}K(t,s)+\int\limits_{t_{0}}^{t}\prescript{RL,t}{s}{D^{\alpha}_{t}}K(t,s)\mathrm{d}s, \quad  t \in \mathbb{\hat{J}}.
\end{equation}
\end{thm}
\begin{proof}
	Using the Definition \ref{defRL} and Fubini's theorem \cite{Chen}, we have
\allowdisplaybreaks
\begin{align*}
\prescript{RL}{t_{0}}{D^{\alpha}_{t}}\int\limits_{t_{0}}^{t}K(t,s)\mathrm{d}s&=\frac{1}{\Gamma(n-\alpha)}\frac{d^{n}}{dt^{n}}\int\limits_{t_{0}}^{t}(t-\tau)^{n-\alpha-1}\mathrm{d}\tau \int\limits_{t_{0}}^{\tau}K(\tau,s)\mathrm{d}s\\
&=\frac{1}{\Gamma(n-\alpha)}\frac{d^{n}}{dt^{n}}\int\limits_{t_{0}}^{t}\int\limits_{t_{0}}^{\tau}(t-\tau)^{n-\alpha-1}K(\tau,s)\mathrm{d}s\mathrm{d}\tau \\
&=\frac{1}{\Gamma(n-\alpha)}\frac{d^{n}}{dt^{n}}\int\limits_{t_{0}}^{t}\int\limits_{s}^{t}(t-\tau)^{n-\alpha-1}K(\tau,s)\mathrm{d}\tau\mathrm{d}s \\
&=\frac{1}{\Gamma(n-\alpha)}\frac{d^{n}}{dt^{n}}\int\limits_{t_{0}}^{t}\mathrm{d}s\int\limits_{s}^{t}(t-\tau)^{n-\alpha-1}K(\tau,s)\mathrm{d}\tau \\
&=\frac{d^{n}}{dt^{n}}\int\limits_{t_{0}}^{t}\left( \frac{1}{\Gamma(n-\alpha)}\int\limits_{s}^{t}(t-\tau)^{n-\alpha-1}K(\tau,s)\mathrm{d}\tau\right)\mathrm{d}s\\
&=\frac{d^{n}}{dt^{n}}\int\limits_{t_{0}}^{t}\prescript{t}{s}{I^{n-\alpha}_{t}}K(t,s)\mathrm{d}s.
\end{align*}

Using the formula \eqref{eq1} for the last part of above expression, we get a desired result:
\begin{align*}
\prescript{RL}{t_{0}}{D^{\alpha}_{t}}\int\limits_{t_{0}}^{t}K(t,s)\mathrm{d}s&=\sum_{l=1}^{n}\frac{d^{l-1}}{dt^{l-1}}\lim\limits_{s\to t-0}\frac{\partial^{n-l}}{\partial t^{n-l}}\prescript{t}{s}{I^{n-\alpha}_{t}}K(t,s)+\int\limits_{t_{0}}^{t}\frac{\partial^{n}}{\partial t^{n}}\prescript{t}{s}{I^{n-\alpha}_{t}}K(t,s)\mathrm{d}s\\
&=\sum_{l=1}^{n}\frac{d^{l-1}}{dt^{l-1}}\lim\limits_{s\to t-0}\prescript{RL,t}{s}{D^{\alpha-l}_{t}}K(t,s)+\int\limits_{t_{0}}^{t}\prescript{RL,t}{s}{D^{\alpha}_{t}}K(t,s)\mathrm{d}s,\quad t \in \mathbb{\hat{J}}.
\end{align*}
\end{proof}
\begin{coroll}\label{coroll-RL}
If we have $K(t,s)=f(t-s)g(s)$, $t_{0}=0$ and assumptions of Theorem \ref{thm-class} are fulfilled, then following equality holds true for convolution operator in Riemann-Liouville sense for any $n \in\mathbb{N}$:

\begin{align}
\prescript{RL}{0}{D^{\alpha}_{t}}\int\limits_{0}^{t}f(t-s)g(s)\mathrm{d}s&=\sum_{l=1}^{n}\lim\limits_{s\to t-0} \prescript{RL,t}{s}{D^{\alpha-l}_{t}}f(t-s)\frac{d^{l-1}}{dt^{l-1}}\lim\limits_{s\to t-0}g(s)\nonumber\\&+\int\limits_{0}^{t}\prescript{RL,t}{s}{D^{\alpha}_{t}}f(t-s)g(s)\mathrm{d}s, \quad  t >0.
\end{align}
\end{coroll}

	%Again using the Definition \ref{defRL} and Fubini's theorem \cite{Chen} for convolution operator, we attain:
	%\allowdisplaybreaks
%\begin{align*}
%\prescript{RL}{0}{D^{\alpha}_{t}}\int\limits_{0}^{t}f(t-s)g(s)\mathrm{d}s&=\frac{1}{\Gamma(n-\alpha)}\frac{d^{n}}{dt^{n}}\int\limits_{0}^{t}(t-\tau)^{n-\alpha-1}\mathrm{d}\tau \int\limits_{0}^{\tau}f(\tau-s)g(s)\mathrm{d}s\\
%&=\frac{1}{\Gamma(n-\alpha)}\frac{d^{n}}{dt^{n}}\int\limits_{0}^{t}\int\limits_{0}^{\tau}(t-\tau)^{n-\alpha-1} f(\tau-s)g(s)\mathrm{d}s\mathrm{d}\tau\\
%&=\frac{1}{\Gamma(n-\alpha)}\frac{d^{n}}{dt^{n}}\int\limits_{0}^{t}\int\limits_{s}^{t}(t-\tau)^{n-\alpha-1} f(\tau-s)g(s)\mathrm{d}\tau\mathrm{d}s\\
%&=\frac{1}{\Gamma(n-\alpha)}\frac{d^{n}}{dt^{n}}\int\limits_{0}^{t}g(s)\mathrm{d}s\int\limits_{s}^{t}(t-\tau)^{n-\alpha-1} f(\tau-s)\mathrm{d}\tau\\
%&=\frac{d^{n}}{dt^{n}}\int\limits_{0}^{t}\int\limits_{s}^{t}(t-\tau)^{n-\alpha-1} f(\tau-s)g(s)\mathrm{d}\tau\mathrm{d}s\\
%&=\frac{d^{n}}{dt^{n}}\int\limits_{0}^{t}\prescript{t}{s}{I^{n-\alpha}_{t}}f(t-s)g(s)\mathrm{d}s\\
%&=\sum_{l=1}^{n}\lim\limits_{s\to t-0}\frac{\partial^{n-l}}{\partial t^{n-l}} \prescript{t}{s}{I^{n-\alpha}_{t}}f(t-s)\frac{d^{l-1}}{dt^{l-1}}\lim\limits_{s\to t-0}g(s)\\
%&+\int\limits_{0}^{t}\frac{\partial^{n}}{\partial t^{n}} \prescript{t}{s}{I^{n-\alpha}_{t}}f(t-s)g(s)\mathrm{d}s\\
%&=\sum_{l=1}^{n}\lim\limits_{s\to t-0}\frac{\partial^{n-l}}{\partial t^{n-l}} \prescript{RL,t}{s}{D^{\alpha-l}_{t}}f(t-s)\frac{d^{l-1}}{dt^{l-1}}\lim\limits_{s\to t-0}g(s)\\
%&+\int\limits_{0}^{t}\prescript{RL,t}{s}{D^{\alpha}_{t}}f(t-s)g(s)\mathrm{d}s.
%\end{align*}
\begin{proof}
	If we write $f(t-s)g(s)$ instead of $K(t,s)$ in \eqref{thm-RL}, then we obtain
	\begin{align*}
	\prescript{RL}{0}{D^{\alpha}_{t}}\int\limits_{0}^{t}f(t-s)g(s)\mathrm{d}s&=\sum_{l=1}^{n}\frac{d^{l-1}}{dt^{l-1}}\left(\lim\limits_{s \to t-0}\left( \frac{\partial^{n-l}}{\partial t^{n-l}} \prescript{t}{s}{I^{n-\alpha}_{t}}f(t-s)g(s)\right)  \right) \\
	&+\int\limits_{0}^{t}\prescript{RL,t}{s}{D^{\alpha}_{t}}f(t-s)g(s)\mathrm{d}s\\
	&=\sum_{l=1}^{n}\lim\limits_{s \to t-0}\frac{\partial^{n-l}}{\partial t^{n-l}}\prescript{t}{s}{I^{n-\alpha}_{t}}f(t-s)\lim\limits_{s\to t-0}g(s) \\
	&+\int\limits_{0}^{t}\prescript{RL,t}{s}{D^{\alpha}_{t}}f(t-s)g(s)\mathrm{d}s\\
	&=\sum_{l=1}^{n}\lim\limits_{s \to t-0}\prescript{RL,t}{s}{D^{\alpha-l}_{t}}f(t-s)\frac{d^{l-1}}{dt^{l-1}}\lim\limits_{s \to t-0}g(s)\\
	&+\int\limits_{0}^{t}\prescript{RL,t}{s}{D^{\alpha}_{t}}f(t-s)g(s)\mathrm{d}s, \quad t>0.
	\end{align*}
	Thus, the proof is complete.
\end{proof}

Then, we are going to introduce fractional differentiation under the integral sign in Caputo sense which will be useful for checking the candidate solutions of fractional differential equations with multi-orders.
For this, firstly, let us consider partial Caputo fractional differentiation operator of order $n-1<\alpha \leq n$, $n\in\mathbb{N}$ with respect to $t$ of a function $K(t,s)$ of two variables $(t,s)\in\mathbb{J}\times\mathbb{J}$, $K:\mathbb{J}\times\mathbb{J}\to\mathbb{R}$, defined by \cite{Miller-Ross,Oldham-Spanier,Samko-Kilbas-Marichev}:

\begin{equation}
\prescript{C, t}{t_{0}}{D}^{\alpha}_{t}K(t,s) =\prescript{t}{t_{0}}{I}^{n-\alpha}_{t}\frac{\mathrm{\partial}^n}{\mathrm{\partial }t^n}K(t,s) =\frac{1}{\Gamma(n-\alpha)}\int\limits^{t}_{t_{0}}(t-s)^{n-\alpha-1}\frac{\mathrm{\partial }^n}{\mathrm{\partial }s^n}K(s,\tau)\mathrm{d}s, \quad t\in \mathbb{\hat{J}},
\end{equation}

Matychyn and Onyshchenko \cite{Ivan-Viktoriia} showed that the fractional Leibniz integral rule for Caputo fractional derivative coincide with Riemann-Liouville one when $\alpha\in (0, 1]$:
\begin{equation*}
\prescript{C}{t_{0}}{D^{\alpha}_{t}}\int\limits_{t_{0}}^{t}K(t,s)\mathrm{d}s=\lim\limits_{s\to t-0}\prescript{t}{s}{I^{1-\alpha}_{t}}K(t,s)+\int\limits_{t_{0}}^{t}\prescript{RL,t}{s}{D^{\alpha}_{t}}K(t,s)\mathrm{d}s, \quad t \in \mathbb{\hat{J}}.
\end{equation*}
More generally, the fractional Leibniz integral rule for fractional derivative  of order $\alpha \in (n-1,n]$, $n\geq 2$ in Caputo sense is stated and proved in the following theorem.
\allowdisplaybreaks

\begin{thm}
	Let the function $K:\mathbb{J}\times\mathbb{J}\to\mathbb{R}$ be such that the following assumptions are fulfilled.
	
	(a) For every fixed $t\in \mathbb{J}$, the function $\hat{K}(t,s)=\prescript{C,t}{s}{D^{\alpha-1}_{t}}K(t,s)$ is measurable and integrable on $\mathbb{J}$ with respect to some $t^{*}\in \mathbb{J}$;
	
	(b) The partial derivative $\prescript{C,t}{s}{D^{\alpha}_{t}}K(t,s)$ exists for every interior point $(t, s) \in \mathbb{\hat{J}} \times \mathbb{\hat{J}}$;
	
	(c) There exists a non-negative integrable function $g$ such that  $\left|\prescript{C,t}{s}{D^{\alpha}_{t}}K(t,s)\right| \leq g(s)$ for every interior point $(t, s) \in \mathbb{\hat{J}}\times \mathbb{\hat{J}}$;
	
	(d) The integral $\prescript{t}{t_{0}}{I^{n-\alpha}_{t}}\left\lbrace\frac{d^{l-1}}{dt^{l-1}}\lim\limits_{s\to t-0}\frac{\partial^{n-l}}{\partial t^{n-l}}K(t,s)\right\rbrace$ , $l=1,2,\ldots,n$, $n\in\mathbb{N}$ exists for every interior point $(t, s) \in \mathbb{\hat{J}} \times \mathbb{\hat{J}}$;

Then, the following relation holds true for fractional derivative in Caputo sense under Lebesgue integration:
\begin{equation}\label{rel-Cap}
\prescript{C}{t_{0}}{D^{\alpha}_{t}}\int\limits_{t_{0}}^{t}K(t,s)\mathrm{d}s=\prescript{t}{t_{0}}{I^{n-\alpha}_{t}}\left\lbrace \sum_{l=1}^{n}\frac{d^{l-1}}{dt^{l-1}}\lim\limits_{s\to t-0}\frac{\partial^{n-l}}{\partial t^{n-l}}K(t,s)\right\rbrace
+\int\limits_{t_{0}}^{t}\prescript{C,t}{s}{D^{\alpha}_{t}}K(t,s)\mathrm{d}s, \quad t \in \mathbb{\hat{J}}.
\end{equation}
\end{thm}
\begin{proof}
	Using the Definition \ref{defCaputo}, Fubini's theorem \cite{Chen}, and the formula \eqref{eq1} we have
	\allowdisplaybreaks
\begin{align*}
\prescript{C}{t_{0}}{D^{\alpha}_{t}}\int\limits_{t_{0}}^{t}K(t,s)\mathrm{d}s&=\frac{1}{\Gamma(n-\alpha)}\int\limits_{t_{0}}^{t}(t-\tau)^{n-\alpha-1}\mathrm{d}\tau \frac{d^{n}}{d\tau^{n}}\int\limits_{t_{0}}^{\tau}K(\tau,s)\mathrm{d}s\\
&=\frac{1}{\Gamma(n-\alpha)}\int\limits_{t_{0}}^{t}(t-\tau)^{n-\alpha-1}\sum_{l=1}^{n}\frac{d^{l-1}}{d\tau^{l-1}}\lim\limits_{s\to {\tau-0}}\frac{\partial^{n-l}}{\partial\tau^{n-l}}K(\tau,s)\mathrm{d}\tau\\
&+\frac{1}{\Gamma(n-\alpha)}\int\limits_{t_{0}}^{t}(t-\tau)^{n-\alpha-1}\mathrm{d}\tau \int\limits_{t_{0}}^{\tau}\frac{\partial^{n}}{\partial \tau^{n}}K(\tau,s)\mathrm{d}s\\
&=\prescript{t}{t_{0}}{I^{n-\alpha}_{t}}\left\lbrace \sum_{l=1}^{n}\frac{d^{l-1}}{dt^{l-1}}\lim\limits_{s\to t-0}\frac{\partial^{n-l}}{\partial t^{n-l}}K(t,s)\right\rbrace\\
&+\frac{1}{\Gamma(n-\alpha)}\int\limits_{t_{0}}^{t}\int\limits_{t_{0}}^{\tau}(t-\tau)^{n-\alpha-1} \frac{\partial^{n}}{\partial \tau^{n}}K(\tau,s)\mathrm{d}s\mathrm{d}\tau  \\
&=\prescript{t}{t_{0}}{I^{n-\alpha}_{t}}\left\lbrace \sum_{l=1}^{n}\frac{d^{l-1}}{dt^{l-1}}\lim\limits_{s\to t-0}\frac{\partial^{n-l}}{\partial t^{n-l}}K(t,s)\right\rbrace\\
&+\frac{1}{\Gamma(n-\alpha)}\int\limits_{t_{0}}^{t}\int\limits_{s}^{t}(t-\tau)^{n-\alpha-1} \frac{\partial^{n}}{\partial \tau^{n}}K(\tau,s)\mathrm{d}\tau \mathrm{d}s \\
&=\prescript{t}{t_{0}}{I^{n-\alpha}_{t}}\left\lbrace \sum_{l=1}^{n}\frac{d^{l-1}}{dt^{l-1}}\lim\limits_{s\to t-0}\frac{\partial^{n-l}}{\partial t^{n-l}}K(t,s)\right\rbrace\\
&+\frac{1}{\Gamma(n-\alpha)}\int\limits_{t_{0}}^{t}\mathrm{d}s\int\limits_{s}^{t}(t-\tau)^{n-\alpha-1} \frac{\partial^{n}}{\partial \tau^{n}}K(\tau,s)\mathrm{d}\tau  \\
&=\prescript{t}{t_{0}}{I^{n-\alpha}_{t}}\left\lbrace \sum_{l=1}^{n}\frac{d^{l-1}}{dt^{l-1}}\lim\limits_{s\to t-0}\frac{\partial^{n-l}}{\partial t^{n-l}}K(t,s)\right\rbrace\\
&+\int\limits_{t_{0}}^{t}\left( \frac{1}{\Gamma(n-\alpha)}\int\limits_{s}^{t}(t-\tau)^{n-\alpha-1} \frac{\partial^{n}}{\partial \tau^{n}}K(\tau,s)\mathrm{d}\tau\right)\mathrm{d}s   \\
&=\prescript{t}{t_{0}}{I^{n-\alpha}_{t}}\left\lbrace \sum_{l=1}^{n}\frac{d^{l-1}}{dt^{l-1}}\lim\limits_{s\to t-0}\frac{\partial^{n-l}}{\partial t^{n-l}}K(t,s)\right\rbrace\\
&+\int\limits_{t_{0}}^{t} \prescript{t}{s}{I^{n-\alpha}_{t}}\frac{\partial^{n}}{\partial t^{n}}K(t,s)\mathrm{d}s\\
&=\prescript{t}{t_{0}}{I^{n-\alpha}_{t}}\left\lbrace \sum_{l=1}^{n}\frac{d^{l-1}}{dt^{l-1}}\lim\limits_{s\to t-0}\frac{\partial^{n-l}}{\partial t^{n-l}}K(t,s)\right\rbrace\\
&+\int\limits_{t_{0}}^{t}\prescript{C,t}{s}{D^{\alpha}_{t}}K(t,s)\mathrm{d}s,  \quad t \in \mathbb{\hat{J}}.
\end{align*}
Therefore, the proof is compete.
\end{proof}
However, the fractional Leibniz integral rule for Caputo derivative of order $0<\alpha\leq1$ is different from the general case which is given in the relation \eqref{rel-Cap}.
\begin{thm}\label{Leibniz rule-Caputo}
	Let the function $K:\mathbb{J}\times\mathbb{J}\to\mathbb{R}$ be such that the following assumptions are fulfilled.
	
	(a) For every fixed $t\in \mathbb{J}$, the function $\hat{K}(t,s)=\prescript{t}{s}{I^{1-\alpha}_{t}}K(t,s)$ is measurable on $\mathbb{J}$ and integrable on $\mathbb{J}$ with respect to some $t^{*}\in \mathbb{J}$;
	
	(b) The partial derivative $\prescript{RL,t}{s}{D^{\alpha}_{t}}K(t,s)$ exists for every interior point $(t, s) \in \mathbb{\hat{J}} \times \mathbb{\hat{J}}$;
	
	(c) There exists a non-negative integrable function $g$ such that  $\left|\prescript{RL,t}{s}{D^{\alpha}_{t}}K(t,s)\right| \leq g(s)$ for every interior point $(t, s) \in \mathbb{\hat{J}}\times \mathbb{\hat{J}}$;

	Then, the Caputo fractional derivative under Lebesgue integration coincides with the fractional differentiation of an integral in Riemann-Liouville sense for $0<\alpha\leq1$:
	
	\begin{equation}\label{thm-RL-C}
	\prescript{C}{t_{0}}{D^{\alpha}_{t}}\int\limits_{t_{0}}^{t}K(t,s)\mathrm{d}s=\lim\limits_{s\to t-0}\prescript{t}{s}{I^{1-\alpha}_{t}}K(t,s)+\int\limits_{t_{0}}^{t}\prescript{RL,t}{s}{D^{\alpha}_{t}}K(t,s)\mathrm{d}s, \quad  t \in \mathbb{\hat{J}}.
	\end{equation}
\end{thm}

\begin{proof}
In accordance the formula \eqref{relationship}, it is obvious for $\alpha \in (0,1]$:
\begin{equation}\label{formula}
 \left( \prescript{C}{t_{0}}{D^{\alpha}_{t}}g\right) (t)= \left( \prescript{RL}{t_{0}}{D^{\alpha}_{t}}g\right) (t)- \frac{(t-t_{0})^{-\alpha}}{\Gamma(1-\alpha)}g(t_{0}), \quad  t\in\mathbb{\hat{J}}.
\end{equation}
Since the formula \eqref{formula}, we can attain that
\begin{align*}
\prescript{C}{t_{0}}{D^{\alpha}_{t}}\int\limits_{t_{0}}^{t}K(t,s)\mathrm{d}s&=\prescript{RL}{t_{0}}{D^{\alpha}_{t}}\int\limits_{t_{0}}^{t}K(t,s)\mathrm{d}s-\left[ \int\limits_{t_{0}}^{t}K(t,s)\mathrm{d}s\right] _{t=t_{0}}\times\frac{(t-t_{0})^{-\alpha}}{\Gamma(1-\alpha)}\\&=\prescript{RL}{t_{0}}{D^{\alpha}_{t}}\int\limits_{t_{0}}^{t}K(t,s)\mathrm{d}s,\quad t \in \mathbb{\hat{J}}.
\end{align*}
Therefore, fractional Leibniz integral rule for Caputo derivative is identical with the Riemann-Liouville one whenever $0< \alpha \leq 1$:
\begin{equation*}
\prescript{C}{t_{0}}{D^{\alpha}_{t}}\int\limits_{t_{0}}^{t}K(t,s)\mathrm{d}s=\lim\limits_{s\to t-0}\prescript{t}{s}{I^{1-\alpha}_{t}}K(t,s)+\int\limits_{t_{0}}^{t}\prescript{RL,t}{s}{D^{\alpha}_{t}}K(t,s)\mathrm{d}s, \quad  t \in \mathbb{\hat{J}}.
\end{equation*}
\end{proof}
It is important to introduce Caputo fractional derivative of convolution operator in general sense which is so accurate tool for testing particular solution of Caputo type multi-term FDEs.
\begin{coroll}
	If we have $K(t,s)=f(t-s)g(s)$, $t_{0}=0$, and assumptions of Theorem \ref{Leibniz rule-Caputo} are fulfilled, then following equality holds true for convolution operator in Caputo's sense of order $ \alpha \in (n-1,n]$, where $n\geq 2$:
\begin{equation}\label{CFLRC}
\prescript{C}{0}{D^{\alpha}_{t}}\int\limits_{0}^{t}f(t-s)g(s)\mathrm{d}s=\prescript{t}{0}{I^{n-\alpha}_{t}}\left\lbrace  \sum_{l=1}^{n}\lim\limits_{s\to t-0}\frac{\partial^{n-l}}{\partial t^{n-l}}f(t-s)\frac{d^{l-1}}{dt^{l-1}}\lim\limits_{s\to t-0}g(s)\right\rbrace 
+\int\limits_{0}^{t}\prescript{C,t}{s}{D^{\alpha}_{t}}f(\tau-s)g(s)\mathrm{d}s, \quad t>0.
\end{equation}
\end{coroll}

\begin{proof}
	If we write $f(t-s)g(s)$ instead of $K(t,s)$ in \eqref{CFLRC}, then we acquire
	\begin{align*}
	\prescript{C}{0}{D^{\alpha}_{t}}\int\limits_{0}^{t}f(t-s)g(s)\mathrm{d}s&=\prescript{t}{0}{I^{n-\alpha}_{t}}\left\lbrace \sum_{l=1}^{n}\frac{d^{l-1}}{dt^{l-1}}\lim\limits_{s\to t-0}\frac{\partial^{n-l}}{\partial t^{n-l}}f(t-s)g(s)\right\rbrace \\
	&+\int\limits_{0}^{t}\prescript{C,t}{s}{D^{\alpha}_{t}}f(t-s)g(s)\mathrm{d}s\\
	&=\prescript{t}{0}{I^{n-\alpha}_{t}}\left\lbrace \sum_{l=1}^{n}\frac{d^{l-1}}{dt^{l-1}}\left( \lim\limits_{s \to t-0}\frac{\partial^{n-l}}{\partial t^{n-l}}f(t-s)\lim\limits_{s\to t-0}g(s)\right) \right\rbrace  \\
	&+\int\limits_{0}^{t}\prescript{C,t}{s}{D^{\alpha}_{t}}f(t-s)g(s)\mathrm{d}s\\
	&=\prescript{t}{0}{I^{n-\alpha}_{t}}\left\lbrace \sum_{l=1}^{n} \lim\limits_{s \to t-0}\frac{\partial^{n-l}}{\partial t^{n-l}}f(t-s)\frac{d^{l-1}}{dt^{l-1}}\lim\limits_{s\to t-0}g(s) \right\rbrace  \\
	&+\int\limits_{0}^{t}\prescript{C,t}{s}{D^{\alpha}_{t}}f(t-s)g(s)\mathrm{d}s, \quad t>0.
	\end{align*}
	Thus, the proof is complete.
\end{proof}

\begin{coroll}
	If we have $K(t,s)=f(t-s)g(s)$, $t_{0}=0$, and assumptions of Theorem \ref{Leibniz rule-Caputo} are fulfilled, then following equality holds true for convolution operator in Caputo's sense for $\alpha \in (0,1]$:
	\begin{equation}\label{special-1}
	\prescript{C}{0}{D^{\alpha}_{t}}\int\limits_{0}^{t}f(t-s)g(s)\mathrm{d}s=\lim\limits_{s\to t-0}\prescript{t}{s}{I^{1-\alpha}_{t}}f(t-s)\lim\limits_{s\to t-0}g(s)+\int\limits_{0}^{t}\prescript{RL,t}{s}{D^{\alpha}_{t}}f(t-s)g(s)\mathrm{d}s, \quad  t > 0.
	\end{equation}
\end{coroll}
\begin{proof}
If we make use of the substitution $K(t,s)=f(t-s)g(s)$ in the relation \eqref{thm-RL-C}, the proof is straightforward. So, we omit it here.
\end{proof}
\begin{thm}\label{relation}
The relationship between  Leibniz integral rule for Riemann-Liouville and Caputo fractional differentiation operators of order $n-1<\alpha\leq n$, $n\geq 2$ holds true:
\begin{equation}
\prescript{RL}{t_{0}}{D^{\alpha}_{t}}\int_{t_{0}}^{t}K(t,s)\mathrm{d}s=\prescript{C}{t_{0}}{D^{\alpha}_{t}}\int_{t_{0}}^{t}K(t,s)\mathrm{d}s+\sum_{i=1}^{n-1}\sum_{l=1}^{i}\left[ \frac{d^{l-1}}{dt^{l-1}}\lim\limits_{s\to t-0}\frac{\partial^{i-l}}{\partial t^{i-l}}K(t,s)\right] _{t=t_{0}}\times\frac{(t-t_{0})^{i-\alpha}}{\Gamma(i-\alpha+1)}, t\in\mathbb{\hat{J}}.
\end{equation}	
\end{thm}
\begin{proof}
	Using the relationship between Riemann-Liouville and Caputo fractional derivatives \eqref{relationship}, we get
	\begin{align}
	\prescript{RL}{t_{0}}{D^{\alpha}_{t}}\int_{t_{0}}^{t}K(t,s)\mathrm{d}s&=
	\prescript{C}{t_{0}}{D^{\alpha}_{t}}\int_{t_{0}}^{t}K(t,s)\mathrm{d}s+\sum_{i=0}^{n-1}\frac{(t-t_{0})^{i-\alpha}}{\Gamma(i-\alpha+1)}\Big[\frac{d^{i}}{d t^{i}}\int_{t_{0}}^{t}K(t,s)\mathrm{d}s\Big]_{t=t_{0}}\nonumber\\
	&=\prescript{C}{t_{0}}{D^{\alpha}_{t}}\int_{t_{0}}^{t}K(t,s)\mathrm{d}s+\sum_{i=1}^{n-1}\frac{(t-t_{0})^{i-\alpha}}{\Gamma(i-\alpha+1)}\Big[\sum_{l=1}^{i}\frac{d^{l-1}}{dt^{l-1}}\lim\limits_{s\to t-0}\frac{\partial^{i-l}}{\partial t^{i-l}}K(t,s)\Big]_{t=t_{0}}\nonumber\\
	&+\frac{(t-t_{0})^{-\alpha}}{\Gamma(1-\alpha)}\left[ \int_{t_{0}}^{t}K(t,s)\mathrm{d}s\right] _{t=t_{0}}
	+\sum_{i=0}^{n-1}\frac{(t-t_{0})^{i-\alpha}}{\Gamma(i-\alpha+1)}\Big[\int_{t_{0}}^{t}\frac{\partial^{i}}{\partial t^{i}}K(t,s)\mathrm{d}s\Big]_{t=t_{0}}\nonumber\\
	&=\prescript{C}{t_{0}}{D^{\alpha}_{t}}\int_{t_{0}}^{t}K(t,s)\mathrm{d}s+\sum_{i=1}^{n-1}\sum_{l=1}^{i}\left[ \frac{d^{l-1}}{dt^{l-1}}\lim\limits_{s\to t-0}\frac{\partial^{i-l}}{\partial t^{i-l}}K(t,s)\right] _{t=t_{0}}\times\frac{(t-t_{0})^{i-\alpha}}{\Gamma(i-\alpha+1)}, t \in \mathbb{\hat{J}}.
	\end{align}
\end{proof}

\begin{coroll}
 If we replace $K(t,s)$ with $f(t-s)g(s)$ and consider $t_{0}=0$ for lower bound of the integral in Theorem \ref{relation}, the relationship between Riemann-Liouville and Caputo type Leibniz integral rules for convolution operator of the functions $f$ and $g$ holds true for $n-1<\alpha\leq n, n\geq2$:
\end{coroll}
\begin{align}
 \prescript{RL}{0}{D^{\alpha}_{t}}\int_{0}^{t}f(t-s)g(s)\mathrm{d}s&=
 \prescript{C}{0}{D^{\alpha}_{t}}\int_{0}^{t}f(t-s)g(s)\mathrm{d}s\nonumber\\
 &+\sum_{i=1}^{n-1}\sum_{l=1}^{i}\lim\limits_{s\to t-0}\frac{\partial^{i-l}}{\partial t^{i-l}}f(t-s)\left[ \frac{d^{l-1}}{dt^{l-1}}\lim\limits_{s\to t-0}g(s)\right] _{t=0}\frac{t^{-\alpha}}{\Gamma(1-\alpha)}, \quad t>0.
 \end{align}

\begin{coroll}
	The fractional Leibniz rule for Riemann-Liouville and Caputo type fractional differential operators coincides for  $0<\alpha\leq1$:
	\begin{equation}
	\prescript{RL}{t_{0}}{D^{\alpha}_{t}}\int_{t_{0}}^{t}K(t,s)\mathrm{d}s=\prescript{C}{t_{0}}{D^{\alpha}_{t}}\int_{t_{0}}^{t}K(t,s)\mathrm{d}s, \quad t\in\mathbb{\hat{J}},
	\end{equation}
	\begin{equation}
	\prescript{RL}{0}{D^{\alpha}_{t}}\int_{0}^{t}f(t-s)g(s)\mathrm{d}s=\prescript{C}{0}{D^{\alpha}_{t}}\int_{0}^{t}f(t-s)g(s)\mathrm{d}s, \quad t>0.
	\end{equation}
\end{coroll}
 \section{Fractional Green's function method}\label{Green's function method}
 The Laplace transform is a convenient technique for solving the Cauchy problem associated with multi-term FDEs with constant coefficients. For instance, let us consider linear in-homogeneous FDE with multi-orders in Caputo's sense and constant coefficients:
 \begin{equation}\label{1}
 \left\lbrace \prescript{C}{0}{D}^{\alpha_{n}}_{t}+\lambda_{1}\prescript{C}{0}{D}^{\alpha_{n-1}}_{t}+\lambda_{2}\prescript{C}{0}{D}^{\alpha_{n-2}}_{t}+\ldots+\lambda_{n-1}\prescript{C}{0}{D}^{\alpha_{1}}_{t}+\lambda_{n}\right\rbrace y(t)=g(t), \quad t>0,
 \end{equation}
 under the homogeneous initial conditions:
 \begin{equation}\label{2}
 y^{(k)}(0)=0, \quad k=0,1,\ldots,n-1,
 \end{equation}
 where $\prescript{C}{0}{D}^{\alpha_{i}}_{(\cdot)}y(\cdot),   i=1,2,\ldots,n$, are the Caputo fractional differentiation operators of orders  $i-1\leq\alpha_{i}\leq i$, \quad $\lambda_{i} \in \mathbb{R}$ for $i=1,2,\ldots,n$ denote constants and $g\in C([0,\infty),\mathbb{R})$ is the continuous force or input function.
 
 The classical analogue of the same problem is considered by Miller in \cite{Miller}. Let us consider IVP for $n$-th order linear differential equation with constant coefficients:
 \begin{equation}\label{D}
 \left\lbrace D^{n}+\lambda_{1}D^{n-1}+\lambda_{2}D^{n-2}+\ldots+\lambda_{n-1}D+\lambda_{n}\right\rbrace y(t)=0, \quad t>0,
 \end{equation}
 with zero initial conditions
 \begin{equation}
 D^{k}y(0)=0, \quad 0 \leq k \leq n-1.
 \end{equation}
 The fractional Green's function is a very useful and applicable practical as well as theoretical tool for solving the IVP (\ref{1})-(\ref{2}) for multi-order FDE.
 
 If we let 
 
\begin{equation}
P(x)=x^{n}+\lambda_{1}x^{n-1}+\lambda_{2}x^{n-2}+\ldots\lambda_{n-1}x+\lambda_{n}
\end{equation}
be a polynomial which is related to the equation \eqref{D},
then according to the Laplace transform method, the unique solution of the following differential system:
\begin{equation}
\begin{cases}
P(D)y(t)=g(t), \quad t>0,\\\nonumber
D^{k}y(0)=0, \quad k=0,1,\ldots,n-1,
\end{cases}
\end{equation}
can be represented in terms of a convolution integral
\begin{equation}
y(t)=\int\limits_{0}^{t}H(t-s)g(s)\mathrm{d}s, \quad t>0,
\end{equation}
where $H(\cdot)$ is the Green or weight function associated with the differential operator $P(D)$ that is evaluated by taking inverse Laplace transform of the transfer function.

\subsection{Applications of fractional Leibniz rules}
In this subsection, we study applications of the  fractional Leibniz integral rule in Riemann-Liouville and Caputo sense using the generalized Bagley-Torvik equations. Moreover, we have used  Leibniz integral rule for checking candidate solution of the oscillator equation in classical sense.

In the following cases, to obtain analytical representation of solutions for the Cauchy problem we will apply fractional Green's function method as we mentioned in Section \ref{Green's function method}.

\textbf{Case 1:}
 We consider the IVP for generalized Bagley-Torvik equations with Riemann-Liouville fractional derivatives of order $1<\alpha \leq 2$ and $0<\beta \leq 1$ in the form of:
\begin{align}\label{RiLi}
\begin{cases}
\left( \prescript{RL}{0}{D}^{\alpha}_{t}y\right) (t)-\mu \left( \prescript{RL}{0}{D}^{\beta}_{t}y\right) (t)-\lambda y(t)=g(t),\quad t>0,\\
\prescript{}{0}{I^{2-\alpha}_{t}}y(t)|_{t=0}=\prescript{}{0}{I^{1-\alpha}_{t}}y(t)|_{t=0}=0, \quad  \quad \lambda,\mu \in \mathbb{R}.
\end{cases}
\end{align}

\begin{thm}\label{thm:main}
	A unique solution $y \in C^{2}([0,\infty) ,\mathbb{R})$ of the Cauchy problem \eqref{RiLi} has the following formula:
	\begin{equation}\label{solRiLi}
	y(t)=\int\limits_{0}^{t}(t-s)^{\alpha-1}E_{\alpha,\alpha-\beta,\alpha}(\lambda(t-s)^{\alpha},\mu(t-s)^{\alpha-\beta})g(s)\mathrm{d}s.
	\end{equation}
\end{thm}
\begin{proof}
	We assume that \eqref{RiLi} has a unique solution $y(t)$ and $g(t)$ is continuous on $\left[ 0, \infty\right)$ and exponentially bounded, then $y(t)$, $\left( \prescript{RL}{0}D^{\alpha}_{t} y\right) (t)$, and $\left( \prescript{RL}{0}D^{\beta}_{t} y\right) (t)$ are exponentially bounded, thus their Laplace transform exist.
	
	Applying Laplace integral transform for Riemann-Liouville fractional derivative using the formula \eqref{lap-RL} to the both sides of \eqref{RiLi} yields:
	\begin{align}\label{eq:f2}
	\left( s^{\alpha}-\mu s^{\beta}-\lambda\right)Y(s)=G(s).
	\end{align}
	Then we solve \eqref{eq:f2} with respect to $Y(s)$,
	\allowdisplaybreaks
	\begin{align}\label{eq:f3}
	Y(s)=\frac{G(s)}{s^{\alpha}-\mu s^{\beta}-\lambda}. 
    \end{align}
	Taking inverse Laplace transform of \eqref{eq:f3} and applying Lemma \ref{lem:2}, we find an explicit representation of solution to \eqref{RiLi}:
	\begin{align} \label{lastpart}
	y(t)=\int\limits_{0}^{t}(t-s)^{\alpha-1}E_{\alpha,\alpha-\beta,\alpha}(\lambda (t-s)^{\alpha}, \mu (t-s)^{\alpha-\beta})g(s)\mathrm{d}s.
	\end{align}
\end{proof}
\textbf{Verification by substitution}. Having found explicit form for $y(t)$, it remains to confirm that $y(t)$ is an analytical solution of \eqref{RiLi} indeed. 
Firstly, for make the use of checking by substitution, we apply fractional  Leibniz integral rule in Riemann-Liouville sense for the first and second terms of \eqref{RiLi}. Then the first term will be as follows:

\begin{align*}
\left( \prescript{RL}{0}{D}^{\alpha}_{t}y\right) (t)&=\prescript{RL}{0}{D}^{\alpha}_{t}\int\limits_{0}^{t}(t-s)^{\alpha-1}E_{\alpha,\alpha-\beta,\alpha}(\lambda(t-s)^{\alpha},\mu(t-s)^{\alpha-\beta})g(s)\mathrm{d}s\\
&=\sum_{l=1}^{2}\lim\limits_{s\to t-0}\prescript{RL,t}{0}{D}^{\alpha-l}_{t}(t-s)^{\alpha-1}E_{\alpha,\alpha-\beta,\alpha}(\lambda(t-s)^{\alpha},\mu(t-s)^{\alpha-\beta})\frac{d^{l-1}}{dt^{l-1}}\lim\limits_{s\to t-0}g(s)\\
&+\int\limits_{0}^{t}\prescript{RL,t}{0}{D}^{\alpha}_{t}(t-s)^{\alpha-1}E_{\alpha,\alpha-\beta,\alpha}(\lambda(t-s)^{\alpha},\mu(t-s)^{\alpha-\beta})g(s)\mathrm{d}s\\
&=\lim\limits_{s\to t-0}\prescript{RL,t}{0}{D}^{\alpha-1}_{t}(t-s)^{\alpha-1}E_{\alpha,\alpha-\beta,\alpha}(\lambda(t-s)^{\alpha},\mu(t-s)^{\alpha-\beta})\lim\limits_{s\to t-0}g(s)\\
&+\lim\limits_{s\to t-0}\prescript{RL,t}{0}{D}^{\alpha-2}_{t}(t-s)^{\alpha-1}E_{\alpha,\alpha-\beta,\alpha}(\lambda(t-s)^{\alpha},\mu(t-s)^{\alpha-\beta})\frac{d}{dt}\lim\limits_{s\to t-0}g(s)\\
&+\int\limits_{0}^{t}\prescript{RL,t}{0}{D}^{\alpha}_{t}(t-s)^{\alpha-1}E_{\alpha,\alpha-\beta,\alpha}(\lambda(t-s)^{\alpha},\mu(t-s)^{\alpha-\beta})g(s)\mathrm{d}s\\
&=\lim\limits_{s\to t-0}E_{\alpha,\alpha-\beta,1}(\lambda(t-s)^{\alpha},\mu(t-s)^{\alpha-\beta})\lim\limits_{s\to t-0}g(s)\\
&+\lim\limits_{s\to t-0}(t-s)E_{\alpha,\alpha-\beta,2}(\lambda(t-s)^{\alpha},\mu(t-s)^{\alpha-\beta})\frac{d}{dt}\lim\limits_{s\to t-0}g(s)\\
&+\int\limits_{0}^{t}\prescript{RL,t}{0}{D}^{\alpha}_{t}(t-s)^{\alpha-1}E_{\alpha,\alpha-\beta,\alpha}(\lambda(t-s)^{\alpha},\mu(t-s)^{\alpha-\beta})g(s)\mathrm{d}s.
\end{align*}
From now on, we apply Pascal's rule for binomial coefficients to the first term of above expression and the limit of the second term is equal to zero as $s\to t-0$, we obtain

\begin{align*}
\left( \prescript{RL}{0}{D}^{\alpha}_{t}y\right) (t)&=\prescript{RL}{0}{D}^{\alpha}_{t}\int\limits_{0}^{t}(t-s)^{\alpha-1}E_{\alpha,\alpha-\beta,\alpha}(\lambda(t-s)^{\alpha},\mu(t-s)^{\alpha-\beta})g(s)\mathrm{d}s\\
&=g(t)+\lambda\lim\limits_{s\to t-0}(t-s)^{\alpha}E_{\alpha,\alpha-\beta,\alpha+1}(\lambda(t-s)^{\alpha},\mu(t-s)^{\alpha-\beta})\lim\limits_{s\to t-0}g(s)\\
&+\mu \lim\limits_{s\to t-0}(t-s)^{\alpha-\beta}E_{\alpha,\alpha-\beta,\alpha-\beta+1}(\lambda(t-s)^{\alpha},\mu(t-s)^{\alpha-\beta})\lim\limits_{s\to t-0}g(s)\\
&+\int\limits_{0}^{t}\prescript{RL,t}{0}{D}^{\alpha}_{t}(t-s)^{\alpha-1}E_{\alpha,\alpha-\beta,\alpha}(\lambda(t-s)^{\alpha},\mu(t-s)^{\alpha-\beta})g(s)\mathrm{d}s\\
&=g(t)+\int\limits_{0}^{t}\prescript{RL,t}{0}{D}^{\alpha}_{t}(t-s)^{\alpha-1}E_{\alpha,\alpha-\beta,\alpha}(\lambda(t-s)^{\alpha},\mu(t-s)^{\alpha-\beta})g(s)\mathrm{d}s.
\end{align*}
Now, using Lemma \ref{Lem:biML-RL} and again applying Pascal's rule for the expression under above integral, we attain
\begin{align*}
&\int\limits_{0}^{t}\prescript{RL,t}{0}{D}^{\alpha}_{t}(t-s)^{\alpha-1}E_{\alpha,\alpha-\beta,\alpha}(\lambda(t-s)^{\alpha},\mu(t-s)^{\alpha-\beta})g(s)\mathrm{d}s\\
=&\int\limits_{0}^{t}(t-s)^{-1}E_{\alpha,\alpha-\beta,0}(\lambda(t-s)^{\alpha},\mu(t-s)^{\alpha-\beta})g(s)\mathrm{d}s\\
=&\int\limits_{0}^{t}\sum_{l=0}^{\infty}\sum_{k=0}^{\infty}\binom{l+k}{k}\frac{\lambda^{l}\mu^{k}(t-s)^{l\alpha+k(\alpha-\beta)-1}}{\Gamma(l\alpha+k(\alpha-\beta))}g(s)\mathrm{d}s\\
=&\int\limits_{0}^{t}\frac{(t-s)^{-1}}{\Gamma(0)}g(s)\mathrm{d}s+\int\limits_{0}^{t}\sum_{l=1}^{\infty}\sum_{k=0}^{\infty}\binom{l+k-1}{k}\frac{\lambda^{l}\mu^{k}(t-s)^{l\alpha+k(\alpha-\beta)-1}}{\Gamma(l\alpha+k(\alpha-\beta))}g(s)\mathrm{d}s\\
+&\int\limits_{0}^{t}\sum_{l=0}^{\infty}\sum_{k=1}^{\infty}\binom{l+k-1}{k-1}\frac{\lambda^{l}\mu^{k}(t-s)^{l\alpha+k(\alpha-\beta)-1}}{\Gamma(l\alpha+k(\alpha-\beta))}g(s)\mathrm{d}s\\
=&\int\limits_{0}^{t}\sum_{l=0}^{\infty}\sum_{k=0}^{\infty}\binom{l+k}{k}\frac{\lambda^{l+1}\mu^{k}(t-s)^{(l+1)\alpha+k(\alpha-\beta)-1}}{\Gamma((l+1)\alpha+k(\alpha-\beta))}g(s)\mathrm{d}s\\
+&\int\limits_{0}^{t}\sum_{l=0}^{\infty}\sum_{k=0}^{\infty}\binom{l+k}{k}\frac{\lambda^{l}\mu^{k+1}(t-s)^{l\alpha+(k+1)(\alpha-\beta)-1}}{\Gamma(l\alpha+(k+1)(\alpha-\beta))}g(s)\mathrm{d}s\\
=&\lambda\int\limits_{0}^{t}(t-s)^{\alpha-1}E_{\alpha,\alpha-\beta,\alpha}(\lambda(t-s)^{\alpha},\mu(t-s)^{\alpha-\beta})g(s)\mathrm{d}s\\
+&\mu\int\limits_{0}^{t}(t-s)^{\alpha-\beta-1}E_{\alpha,\alpha-\beta,\alpha-\beta}(\lambda(t-s)^{\alpha},\mu(t-s)^{\alpha-\beta})g(s)\mathrm{d}s.
\end{align*}

Therefore, we have 
\begin{align}\label{alpha}
\left( \prescript{RL}{0}{D}^{\alpha}_{t}y\right)(t) 
&=g(t)+\lambda\int\limits_{0}^{t}(t-s)^{\alpha-1}E_{\alpha,\alpha-\beta,\alpha}(\lambda(t-s)^{\alpha},\mu(t-s)^{\alpha-\beta})g(s)\mathrm{d}s\nonumber\\
&+\mu\int\limits_{0}^{t}(t-s)^{\alpha-\beta-1}E_{\alpha,\alpha-\beta,\alpha-\beta}(\lambda(t-s)^{\alpha},\mu(t-s)^{\alpha-\beta})g(s)\mathrm{d}s
.
\end{align}
Similarly, the second term of \eqref{RiLi} will  be
\begin{align}\label{beta}
&\left( \prescript{RL}{0}{D}^{\beta}_{t}y\right)(t)= \prescript{RL}{0}{D}^{\beta}_{t}\int\limits_{0}^{t}(t-s)^{\alpha-1}E_{\alpha,\alpha-\beta,\alpha}(\lambda(t-s)^{\alpha},\mu(t-s)^{\alpha-\beta})g(s)\mathrm{d}s\nonumber\\
&=\int\limits_{0}^{t}(t-s)^{\alpha-\beta-1}E_{\alpha,\alpha-\beta,\alpha-\beta}(\lambda(t-s)^{\alpha},\mu(t-s)^{\alpha-\beta})g(s)\mathrm{d}s.
\end{align}

Taking linear combination of and \eqref{alpha} and \eqref{beta} together with \eqref{solRiLi}, we get the desired result.

 \textbf{Case 2:} 
 We consider the Cauchy problem for generalized Bagley-Torvik equations which is the special case of Caputo type fractional multi-term differential equations with constant coefficients of order $1<\alpha \leq 2$ and $0<\beta \leq 1$ in the form of:
 \begin{align}\label{caputo}
 \begin{cases}
 \left( \prescript{C}{0}{D}^{\alpha}_{t}y\right) (t)-\mu \left( \prescript{C}{0}{D}^{\beta}_{t}y\right) (t)-\lambda y(t)=g(t),\quad t>0,\\
 y(0)=y^{\prime}(0)=0, \quad  \quad \lambda,\mu \in \mathbb{R}.
 \end{cases}
 \end{align}
\begin{thm}\label{thm:main1}
	A unique solution $y \in C^{2}([0,\infty) ,\mathbb{R})$ of the Cauchy problem \eqref{caputo} has the following formula:
	\begin{equation}\label{solC}
	y(t)=\int\limits_{0}^{t}(t-s)^{\alpha-1}E_{\alpha,\alpha-\beta,\alpha}(\lambda(t-s)^{\alpha},\mu(t-s)^{\alpha-\beta})g(s)\mathrm{d}s.
	\end{equation}
\end{thm}
\begin{proof}
	We assume that \eqref{caputo} has a unique solution $y(t)$ and $g(t)$ is continuous on $\left[ 0, \infty\right)$ and exponentially bounded, then $y(t)$, $\left( \prescript{C}{0}D^{\alpha}_{t} y\right) (t)$, and $\left( \prescript{C}{0}D^{\beta}_{t} y\right) (t)$ are exponentially bounded, thus their Laplace transform exist.
	
	Applying the formula of Laplace transform for Caputo fractional derivative \eqref{lap-C} to the both sides of \eqref{caputo} yields:
	\begin{align}\label{eq:f}
	\left( s^{\alpha}-\mu s^{\beta}-\lambda\right)Y(s)=G(s).
	\end{align}
	Then we solve \eqref{eq:f} with respect to $Y(s)$,
	\allowdisplaybreaks
	\begin{align}\label{eq:f5}
	Y(s)=\frac{G(s)}{s^{\alpha}-\mu s^{\beta}-\lambda}. 
	\end{align}
	Taking inverse Laplace transform of \eqref{eq:f5} and applying Lemma \ref{lem:2}, we find an explicit representation of solution to \eqref{caputo}:
	\begin{align} \label{lastpart1}
	y(t)=\int\limits_{0}^{t}(t-s)^{\alpha-1}E_{\alpha,\alpha-\beta,\alpha}(\lambda (t-s)^{\alpha}, \mu (t-s)^{\alpha-\beta})g(s)\mathrm{d}s.
	\end{align}
\end{proof}
\begin{remark}
Since initial conditions equal to zero,  in accordance with the formula \eqref{relationship}, analytical solutions should be coincide with each other for the Cauchy problems in Riemann-Liouville \eqref{RiLi} and Caputo \eqref{caputo} senses.
\end{remark}
\textbf{Verification by substitution}. Having found explicit form for $y(t)$, it remains to confirm that $y(t)$ is an analytical solution of \eqref{caputo} indeed.

Now, we again make use of checking by substitution via Caputo fractional Leibniz integral rule. In this case, we apply first Pascal's rule before applying fractional Leibniz rule since $\prescript{C}{s}{D}^{\alpha}_{t}\left( \frac{(t-s)^{\alpha-1}}{\Gamma(\alpha)}\right) $ is undefined in accordance with \eqref{Cap}. Then according to the formula \eqref{R-L der-R-L int}, we obtain
\begin{align}\label{main2}
&\left( \prescript{C}{0}{D}^{\alpha}_{t}y\right) (t)=\prescript{C}{0}{D}^{\alpha}_{t}\int\limits_{0}^{t}(t-s)^{\alpha-1}E_{\alpha,\alpha-\beta,\alpha}(\lambda(t-s)^{\alpha},\mu(t-s)^{\alpha-\beta})g(s)\mathrm{d}s\nonumber\\&=g(t)+\lambda\prescript{C}{0}{D}^{\alpha}_{t}\int\limits_{0}^{t}(t-s)^{2\alpha-1}E_{\alpha,\alpha-\beta,2\alpha}(\lambda(t-s)^{\alpha},\mu(t-s)^{\alpha-\beta})g(s)\mathrm{d}s\nonumber\\
&+\mu\prescript{C}{0}{D}^{\alpha}_{t}\int\limits_{0}^{t}(t-s)^{2\alpha-\beta-1}E_{\alpha,\alpha-\beta,2\alpha-\beta}(\lambda(t-s)^{\alpha},\mu(t-s)^{\alpha-\beta})g(s)\mathrm{d}s.
\end{align}

Then using the formula for fractional Leibniz integral rule in Caputo sense \eqref{CFLRC} of order $1<\alpha\leq2$, we have 

\begin{align*}
&\prescript{C}{0}{D}^{\alpha}_{t}\int\limits_{0}^{t}(t-s)^{2\alpha-1}E_{\alpha,\alpha-\beta,2\alpha}(\lambda(t-s)^{\alpha},\mu(t-s)^{\alpha-\beta})g(s)\mathrm{d}s\nonumber\\
&=\prescript{t}{0}{I^{2-\alpha}_{t}}\Big[\lim\limits_{s\to t-0}\sum_{l=1}^{2}\frac{\partial^{2-l}}{\partial t^{2-l}}(t-s)^{2\alpha-1}E_{\alpha,\alpha-\beta,2\alpha}(\lambda(t-s)^{\alpha},\mu(t-s)^{\alpha-\beta})\frac{d^{l-1}}{dt^{l-1}}\lim\limits_{s\to t-0}g(s)\Big]\nonumber\\
&+\int\limits_{0}^{t}\prescript{C,t}{0}{D}^{\alpha}_{t}(t-s)^{2\alpha-1}E_{\alpha,\alpha-\beta,2\alpha}(\lambda(t-s)^{\alpha},\mu(t-s)^{\alpha-\beta})g(s)\mathrm{d}s\nonumber\\
&=\prescript{t}{0}{I^{2-\alpha}_{t}}\Big[\lim\limits_{s\to t-0}\frac{\partial}{\partial t}(t-s)^{2\alpha-1}E_{\alpha,\alpha-\beta,2\alpha}(\lambda(t-s)^{\alpha},\mu(t-s)^{\alpha-\beta})\lim\limits_{s\to t-0}g(s)\Big]\\
&+\prescript{t}{0}{I^{2-\alpha}_{t}}\Big[\lim\limits_{s\to t-0}(t-s)^{2\alpha-1}E_{\alpha,\alpha-\beta,2\alpha}(\lambda(t-s)^{\alpha},\mu(t-s)^{\alpha-\beta})\frac{d}{dt}\lim\limits_{s\to t-0}g(s)\Big]\\
&+\int\limits_{0}^{t}\prescript{C,t}{0}{D}^{\alpha}_{t}(t-s)^{2\alpha-1}E_{\alpha,\alpha-\beta,2\alpha}(\lambda(t-s)^{\alpha},\mu(t-s)^{\alpha-\beta})g(s)\mathrm{d}s\nonumber\\
&=\prescript{t}{0}{I^{2-\alpha}_{t}}\Big[\lim\limits_{s\to t-0}(t-s)^{2\alpha-2}E_{\alpha,\alpha-\beta,2\alpha-1}(\lambda(t-s)^{\alpha},\mu(t-s)^{\alpha-\beta})\lim\limits_{s\to t-0}g(s)\Big]\\
&+\prescript{t}{0}{I^{2-\alpha}_{t}}\Big[\lim\limits_{s\to t-0}(t-s)^{2\alpha-1}E_{\alpha,\alpha-\beta,2\alpha}(\lambda(t-s)^{\alpha},\mu(t-s)^{\alpha-\beta})\frac{d}{dt}\lim\limits_{s\to t-0}g(s)\Big]\\
&+\int\limits_{0}^{t}\prescript{C,t}{0}{D}^{\alpha}_{t}(t-s)^{2\alpha-1}E_{\alpha,\alpha-\beta,2\alpha}(\lambda(t-s)^{\alpha},\mu(t-s)^{\alpha-\beta})g(s)\mathrm{d}s\nonumber\\
&=\int\limits_{0}^{t}\prescript{C,t}{0}{D}^{\alpha}_{t}(t-s)^{2\alpha-1}E_{\alpha,\alpha-\beta,2\alpha}(\lambda(t-s)^{\alpha},\mu(t-s)^{\alpha-\beta})g(s)\mathrm{d}s.
\end{align*}

Thus, by Lemma \ref{Lem:biML}, we get 
\begin{align}\label{221}
	&\prescript{C}{0}{D}^{\alpha}_{t}\int\limits_{0}^{t}(t-s)^{2\alpha-1}E_{\alpha,\alpha-\beta,2\alpha}(\lambda(t-s)^{\alpha},\mu(t-s)^{\alpha-\beta})g(s)\mathrm{d}s\nonumber\\
	&=\int\limits_{0}^{t}\prescript{C,t}{0}{D}^{\alpha}_{t}(t-s)^{2\alpha-1}E_{\alpha,\alpha-\beta,2\alpha}(\lambda(t-s)^{\alpha},\mu(t-s)^{\alpha-\beta})g(s)\mathrm{d}s\nonumber\\
	&=\int\limits_{0}^{t}(t-s)^{\alpha-1}E_{\alpha,\alpha-\beta,\alpha}(\lambda(t-s)^{\alpha},\mu(t-s)^{\alpha-\beta})g(s)\mathrm{d}s.
\end{align}
Similarly, by applying the formula \eqref{CFLRC}, we also have
\begin{align}\label{222}
	&\prescript{C}{0}{D}^{\alpha}_{t}\int\limits_{0}^{t}(t-s)^{2\alpha-\beta-1}E_{\alpha,\alpha-\beta,2\alpha-\beta}(\lambda(t-s)^{\alpha},\mu(t-s)^{\alpha-\beta})g(s)\mathrm{d}s\nonumber\\
	&=\int\limits_{0}^{t}\prescript{C,t}{0}{D}^{\alpha}_{t}(t-s)^{2\alpha-\beta-1}E_{\alpha,\alpha-\beta,2\alpha-\beta}(\lambda(t-s)^{\alpha},\mu(t-s)^{\alpha-\beta})g(s)\mathrm{d}s\nonumber\\
	&=\int\limits_{0}^{t}(t-s)^{\alpha-\beta-1}E_{\alpha,\alpha-\beta,\alpha-\beta}(\lambda(t-s)^{\alpha},\mu(t-s)^{\alpha-\beta})g(s)\mathrm{d}s.
\end{align}
 
On the other hand, since $0< \beta \leq 1$, we will apply the formula \eqref{special-1} for second term of the equation \eqref{caputo}:

\begin{align}\label{betCFL}
	&\prescript{C}{0}{D}^{\beta}_{t}\int\limits_{0}^{t}(t-s)^{\alpha-1}E_{\alpha,\alpha-\beta,\alpha}(\lambda(t-s)^{\alpha},\mu(t-s)^{\alpha-\beta})g(s)\mathrm{d}s\nonumber\\
	&=\lim\limits_{s\to t-0}\prescript{t}{0}{I}^{1-\beta}_{t}(t-s)^{\alpha-1}E_{\alpha,\alpha-\beta,\alpha}(\lambda(t-s)^{\alpha},\mu(t-s)^{\alpha-\beta})\lim\limits_{s\to t-0}g(s)\nonumber\\
	&+\int\limits_{0}^{t}\prescript{RL,t}{0}{D}^{\beta}_{t}(t-s)^{\alpha-1}E_{\alpha,\alpha-\beta,\alpha}(\lambda(t-s)^{\alpha},\mu(t-s)^{\alpha-\beta})g(s)\mathrm{d}s\nonumber\\
	&=\lim\limits_{s\to t-0}(t-s)^{\alpha-\beta}E_{\alpha,\alpha-\beta,\alpha-\beta+1}(\lambda(t-s)^{\alpha},\mu(t-s)^{\alpha-\beta})\lim\limits_{s\to t-0}g(s)\nonumber\\
	&+\int\limits_{0}^{t}\prescript{RL,t}{0}{D}^{\beta}_{t}(t-s)^{\alpha-1}E_{\alpha,\alpha-\beta,\alpha}(\lambda(t-s)^{\alpha},\mu(t-s)^{\alpha-\beta})g(s)\mathrm{d}s\nonumber\\
	&=\int\limits_{0}^{t}\prescript{RL,t}{0}{D}^{\beta}_{t}(t-s)^{\alpha-1}E_{\alpha,\alpha-\beta,\alpha}(\lambda(t-s)^{\alpha},\mu(t-s)^{\alpha-\beta})g(s)\mathrm{d}s\nonumber\\
	&=\int\limits_{0}^{t}(t-s)^{\alpha-\beta-1}E_{\alpha,\alpha-\beta,\alpha-\beta}(\lambda(t-s)^{\alpha},\mu(t-s)^{\alpha-\beta})g(s)\mathrm{d}s.
\end{align}
where
\begin{equation*}
\prescript{t}{0}{I}^{1-\beta}_{t}\left\lbrace \frac{(t-t_{0})^{\alpha-1}}{\Gamma(\alpha)}\right\rbrace =\frac{(t-t_{0})^{\alpha-\beta}}{\Gamma(\alpha-\beta+1)}, \quad t>0.
\end{equation*}
Next we plug \eqref{221} and \eqref{222} into \eqref{main2}, we therefore get
\begin{align*}
&\prescript{C}{0}{D}^{\alpha}_{t}\int\limits_{0}^{t}(t-s)^{\alpha-1}E_{\alpha,\alpha-\beta,\alpha}(\lambda(t-s)^{\alpha},\mu(t-s)^{\alpha-\beta})g(s)\mathrm{d}s\nonumber\\
&=g(t)+\lambda\int\limits_{0}^{t}(t-s)^{\alpha-1}E_{\alpha,\alpha-\beta,\alpha}(\lambda(t-s)^{\alpha},\mu(t-s)^{\alpha-\beta})g(s)\mathrm{d}s\nonumber\\
&+\mu\int\limits_{0}^{t}(t-s)^{\alpha-\beta-1}E_{\alpha,\alpha-\beta,\alpha-\beta}(\lambda(t-s)^{\alpha},\mu(t-s)^{\alpha-\beta})g(s)\mathrm{d}s.
\end{align*}

Taking linear combination of above equation together with \eqref{solC} and \eqref{betCFL}, we arrive at
\begin{align*}
\left( \prescript{C}{0}{D}^{\alpha}_{t}y\right) (t)-\mu \left( \prescript{C}{0}{D}^{\beta}_{t}y\right) (t)-\lambda y(t)=g(t), \quad t>0.
\end{align*}

\begin{remark}\label{rem2}
Kilbas et al. \cite{Kilbas-Srivastava-Trujillo} have obtained the analytical solutions of	the Cauchy problems \eqref{RiLi} and \eqref{caputo} in terms of Fox-Wright  functions below: 
	\begin{equation}
	y(t)=\int\limits_{0}^{t}(t-s)^{\alpha-1}H_{\alpha,\beta;\lambda,\mu}(t-s)g(s)\mathrm{d}s, \quad t>0,
	\end{equation}
	where
	\begin{equation*}
	H_{\alpha,\beta;\lambda,\mu}(t)\coloneqq\sum_{l=0}^{\infty}\frac{\lambda^{l}t^{l\alpha}}{l!}\prescript{}{1}{\Psi_{1}}\left[\begin{array}{ccc}
	(l+1,1) \\
	(l\alpha+\alpha, \alpha-\beta)
	\end{array}\Big|\mu t^{\alpha-\beta}\right].
	\end{equation*}
\end{remark}

\begin{proof}
Using the definition of Fox-Wright function \cite{Fox,Wright}, we arrive at
\begin{align*}
	y(t)&=\int\limits_{0}^{t}\sum_{l=0}^{\infty}\frac{\lambda^{l}(t-s)^{l\alpha+\alpha-1}}{l!}\prescript{}{1}{\Psi_{1}}\left[\begin{array}{ccc}
	(l+1,1) \\
	(l\alpha+\alpha, \alpha-\beta)
	\end{array}\Big|\mu (t-s)^{\alpha-\beta}\right]g(s)\mathrm{d}s\\
	&=\int\limits_{0}^{t}\sum_{l=0}^{\infty}\sum_{k=0}^{\infty}\binom{l+k}{k}\frac{\lambda^{l}\mu^{k}(t-s)^{l\alpha+k(\alpha-\beta)+\alpha-1}}{\Gamma(l\alpha+k(\alpha-\beta)+\alpha)}g(s)\mathrm{d}s\\
	&=\int\limits_{0}^{t}(t-s)^{\alpha-1}E_{\alpha,\alpha-\beta,\alpha}(\lambda (t-s)^{\alpha}, \mu (t-s)^{\alpha-\beta})g(s)\mathrm{d}s, \quad t>0.
	\end{align*}
	Therefore, our solution in terms of univariate version of bivariate   Mittag-Leffler type functions coincide with the solution by means of Fox-Wright type functions shown in \cite{Kilbas-Srivastava-Trujillo}.  
\end{proof}

\begin{remark}\label{rem1}
	Podlubny  \cite{I. Podlubny} have attained the analytical solutions of	the Cauchy problems \eqref{RiLi} and \eqref{caputo} in terms of $l$-th derivative of two-parameter Mittag-Leffler functions below: 
	\begin{equation}
	y(t)=\int\limits_{0}^{t}(t-s)^{\alpha-1}H_{\alpha,\beta;\lambda,\mu}(t-s)g(s)\mathrm{d}s, \quad t>0,
	\end{equation}
	where
	\begin{equation*}
	H_{\alpha,\beta;\lambda,\mu}(t)\coloneqq\sum_{l=0}^{\infty}\frac{\lambda^{l}t^{l\alpha}}{l!}\mathcal{E}^{(l)}_{\alpha-\beta,\alpha+l\beta}(\mu t^{\alpha-\beta}).
	\end{equation*}
\end{remark}

\begin{proof}
	Using the definition of $l$-th derivative of two-parameter Mittag-Leffler function, we arrive at
	\begin{align*}
	y(t)&=\int\limits_{0}^{t}\sum_{l=0}^{\infty}\frac{\lambda^{l}(t-s)^{l\alpha+\alpha-1}}{l!}\mathcal{E}^{(l)}_{\alpha-\beta,\alpha+l\beta}(\mu(t-s)^{\alpha-\beta})g(s)\mathrm{d}s\\
	&=\int\limits_{0}^{t}\sum_{l=0}^{\infty}\frac{\lambda^{l}}{l!}(t-s)^{(l+1)\alpha-1}\sum_{k=0}^{\infty}\frac{(l+k)!}{k!}\frac{\mu^{k}(t-s)^{k(\alpha-\beta)}}{\Gamma(k(\alpha-\beta)+l(\alpha-\beta)+l\beta+\alpha)}g(s)\mathrm{d}s\\
	&=\int\limits_{0}^{t}\sum_{l=0}^{\infty}\sum_{k=0}^{\infty}\binom{l+k}{k}\frac{\lambda^{l}\mu^{k}(t-s)^{l\alpha+k(\alpha-\beta)+\alpha-1}}{\Gamma(l\alpha+k(\alpha-\beta)+\alpha)}g(s)\mathrm{d}s\\
	&=\int\limits_{0}^{t}(t-s)^{\alpha-1}E_{\alpha,\alpha-\beta,\alpha}(\lambda (t-s)^{\alpha}, \mu (t-s)^{\alpha-\beta})g(s)\mathrm{d}s, \quad t>0.
	\end{align*}
	Therefore, our solution in terms of univariate version of bivariate   Mittag-Leffler type functions coincide with the solution by means of $l$-th derivative of two-parameter Mittag-Leffler type functions shown in \cite{I. Podlubny}.  
\end{proof}
\textbf{Case 3:} 
In special case, we substitute $\alpha=2$ and $\beta=1$ in \eqref{RiLi} and \eqref{caputo}, then we get the following Cauchy problem for the  classical second order linear differential equation - the oscillator equation with constant coefficients:
\begin{equation}\label{classical}
	\begin{cases}
		y^{\prime\prime}(t)-\mu y^{\prime}(t)-\lambda y(t)=g(t), \quad t>0,\\
		y^{\prime}(0)=y(0)=0, \quad \lambda,\mu \in \mathbb{R},
	\end{cases}
\end{equation}

\begin{thm}\label{thm:main3}
	A unique solution $y \in C^{2}([0,\infty) ,\mathbb{R})$ of the Cauchy problem \eqref{classical} has the following formula:
	\begin{equation}\label{sol-class}
y(t)=\int\limits_{0}^{t}(t-s)E_{2,1,2}(\lambda(t-s)^{2},\mu(t-s))g(s)\mathrm{d}s, \quad t>0.
	\end{equation}
\end{thm}
\begin{proof}
By using verification by substitution, we have 
\begin{align}\label{main3}
	&\frac{d^{2}}{dt^{2}}\int\limits_{0}^{t}(t-s)E_{2,1,2}(\lambda(t-s)^{2},\mu(t-s))g(s)\mathrm{d}s\nonumber\\
	&=\sum_{l=1}^{2}\lim\limits_{s\to t-0}\frac{\partial^{2-l}}{\partial t^{2-l}}(t-s)E_{2,1,2}(\lambda(t-s)^{2},\mu(t-s))\frac{d^{l-1}}{dt^{l-1}}\lim\limits_{s\to t-0}g(s)\nonumber\\
	&+\int\limits_{0}^{t}\frac{\partial^{2}}{\partial t^{2}}(t-s)E_{2,1,2}(\lambda(t-s)^{2},\mu(t-s))g(s)\mathrm{d}s\nonumber\\
	&=\lim\limits_{s\to t-0}\frac{\partial}{\partial t}(t-s)E_{2,1,2}(\lambda(t-s)^{2},\mu(t-s))\lim\limits_{s\to t-0}g(s)\nonumber\\
	&+\lim\limits_{s\to t-0}(t-s)E_{2,1,2}(\lambda(t-s)^{2},\mu(t-s))\frac{d}{dt}\lim\limits_{s\to t-0}g(s)\nonumber\\
	&+\int\limits_{0}^{t}\frac{\partial^{2}}{\partial t^{2}}(t-s)E_{2,1,2}(\lambda(t-s)^{2},\mu(t-s))g(s)\mathrm{d}s\nonumber\\
	&=\lim\limits_{s\to t-0}E_{2,1,1}(\lambda(t-s)^{2},\mu(t-s))\lim\limits_{s\to t-0}g(s)\nonumber\\
	&+\int\limits_{0}^{t}(t-s)^{-1}E_{2,1,0}(\lambda(t-s)^{2},\mu(t-s))g(s)\mathrm{d}s\nonumber\\
	&=g(t)+\lambda  \lim\limits_{s\to t-0}(t-s)^{2}E_{2,1,3}(\lambda(t-s)^{2},\mu(t-s))g(s)\mathrm{d}s\nonumber\\
	&+\mu 
	\lim\limits_{s\to t-0}(t-s)E_{2,1,2}(\lambda(t-s)^{2},\mu(t-s))g(s)\mathrm{d}s\nonumber\\
	&+\int\limits_{0}^{t}(t-s)^{-1}E_{2,1,0}(\lambda(t-s)^{2},\mu(t-s))g(s)\mathrm{d}s\nonumber\\
	&=g(t)+\int\limits_{0}^{t}(t-s)^{-1}E_{2,1,0}(\lambda(t-s)^{2},\mu(t-s))g(s)\mathrm{d}s.
\end{align}
Applying Pascal's rule for binomial coefficients for the last term of above equality, we get
\begin{align}\label{alphafirst}
&\int\limits_{0}^{t}(t-s)^{-1}E_{2,1,0}(\lambda(t-s)^{2},\mu(t-s))g(s)\mathrm{d}s\nonumber\\
&=\int\limits_{0}^{t}\frac{(t-s)^{-1}}{\Gamma(0)}g(s)\mathrm{d}s+\lambda\int\limits_{0}^{t}(t-s)E_{2,1,2}(\lambda(t-s)^{2},\mu(t-s))g(s)\mathrm{d}s\nonumber\\
&+\mu\int\limits_{0}^{t}E_{2,1,1}(\lambda(t-s)^{2},\mu(t-s))g(s)\mathrm{d}s\nonumber\\
&=\lambda\int\limits_{0}^{t}(t-s)E_{2,1,2}(\lambda(t-s)^{2},\mu(t-s))g(s)\mathrm{d}s\nonumber\\
&+\mu\int\limits_{0}^{t}E_{2,1,1}(\lambda(t-s)^{2},\mu(t-s))g(s)\mathrm{d}s.
\end{align}

Next, we have
\begin{align}\label{betafirst}
	&\frac{d}{dt}\int\limits_{0}^{t}(t-s)E_{2,1,2}(\lambda(t-s)^{2},\mu(t-s))g(s)\mathrm{d}s\nonumber\\
	&=\int\limits_{0}^{t}\frac{\partial}{\partial t}(t-s)E_{2,1,2}(\lambda(t-s)^{2},\mu(t-s))g(s)\mathrm{d}s\nonumber\\
	&=\int\limits_{0}^{t}E_{2,1,1}(\lambda(t-s)^{2},\mu(t-s))g(s)\mathrm{d}s.
\end{align}

Again taking linear combination of above equations \eqref{alphafirst} and  \eqref{betafirst} together with \eqref{sol-class}, we prove the desired result.
\end{proof}
\section{Conclusions and future work} \label{concl}
The theory of  Leibniz integral rule allows us to study particular solutions of classical and fractional multi-term differential equations. To the best of our knowledge, we derive explicit analytical solutions of well-known Bagley-Torvik and oscillator equations in terms of bivariate Mittag-Leffler functions via the technique of fractional Green's function, since this theory has not been presented in recent literature.

The major contributions of our research work are as below:
\begin{itemize}
	\item we have proposed a  Leibniz rule for higher order derivatives in classical sense which is more productive tool for testing solutions of multi-order differential equation;
	\item we have introduced  fractional Leibniz rule for Riemann-Liouville and Caputo type fractional differentiation operators;
	\item we have investigated differentiation of convolution operator which is more crucial in theory of differential equations with constant coefficients of classical and fractional-order derivatives;
	\item  analytical explicit solutions of the generalized Bagley-Torvik  and oscillator equations are derived in terms of univariate version of bivariate Mittag-Leffler type functions in accordance with the method of Laplace integral transform;
	\item We have showed that our analytical solutions are coincide with Fox-Wright type and $l$-th derivative of two-parameter Mittag-Leffler type functions;
	\item we tested the candidate solutions of Cauchy problems for Bagley-Torvik equations with fractional-order sense and oscillator equation with classical-order one via our new  fractional Leibniz integral rules.
\end{itemize}

There are a number of potential directions in which the results acquired here can be extended. Our future work will proceed to study the Leibniz integral rule results for $\psi$-Hilfer and Hadamard type fractional derivatives and the analytical explicit solutions of multi-term fractional differential equations in terms of natural extensions of Mittag-Leffler type functions.

\end{document}